\documentclass[a4paper, 10pt, leqno]{amsart}
\usepackage{amssymb, amsmath, amsthm, amscd, mathrsfs}

\usepackage{paralist}
\usepackage{cite}
\usepackage{bigints}
\usepackage{dsfont}
\usepackage{empheq}
\usepackage{amssymb}
\usepackage{cases}
\usepackage{enumitem}
\allowdisplaybreaks
\usepackage[foot]{amsaddr}

\usepackage{hyperref}
\usepackage[top=3cm, bottom=2.5cm, left=2.5cm, right=2.5cm]{geometry}
\theoremstyle{plain}
\newtheorem{theorem}{Theorem}[section]

\newtheorem{proposition}[theorem]{Proposition}
\theoremstyle{definition}

\newtheorem{remark}[theorem]{Remark}

\numberwithin{equation}{section}

\def\R{{\mathbb R}}
\def\N{{\mathbb N}}
\renewcommand{\leq}{\leqslant}
\renewcommand{\geq}{\geqslant}
\newcommand{\pt}{\partial_t}
\newcommand{\ptt}{\partial_t^2}

\newcommand{\pnu}{\partial_\nu}
\def \d {\mathrm{d}}

\title[Inverse source problem of the wave] 
      {Lipschitz stability for an inverse source problem of the wave equation with kinetic boundary conditions}

\author{Salah-Eddine Chorfi$^1$}
\author{Ghita El Guermai$^1$}
\author{Lahcen Maniar$^1$}
\author{Walid Zouhair$^2$}

\address[1]{Cadi Ayyad University, UCA, Faculty of Sciences Semlalia, Laboratory of Mathematics, Modeling and Automatic Systems, B.P. 2390, Marrakesh, Morocco, (s.chorfi@uca.ac.ma), (ghita.el.guermai@gmail.com), (maniar@uca.ac.ma)}
\address[2]{Department of Mathematics, Faculty of Applied Sciences, Ibn Zohr University, Azrou, B.P. 6146, Ait Melloul, Morocco, (walid.zouhair.fssm@gmail.com)}

\begin{document}
\begin{abstract}
In this paper, we present a refined approach to establish a global Lipschitz stability for an inverse source problem concerning the determination of forcing terms in the wave equation with mixed boundary conditions. It consists of boundary conditions incorporating a dynamic boundary condition and Dirichlet boundary condition on disjoint subsets of the boundary. The primary contribution of this article is the rigorous derivation of a sharp Carleman estimate for the wave system with a dynamic boundary condition. In particular, our findings complete and drastically improve the earlier results established by Gal and Tebou [SIAM J. Control Optim., 55 (2017), 324--364]. This is achieved by using a different weight function to overcome some relevant difficulties. As for the stability proof, we extend to dynamic boundary conditions a recent argument avoiding cut-off functions. Finally, we also show that our developed Carleman estimate yields a sharp boundary controllability result.


\end{abstract}

\dedicatory{\large Dedicated to the memory of Professor Hammadi Bouslous}

\maketitle

\section{Introduction}
The wave equation $\partial^2_t y(x,t) - d \Delta y(x,t)=0$ is a hyperbolic Partial Differential Equation (PDE) that models the propagation of waves in a medium $\Omega$ and their positions in space $x$ over time $t$, where $\sqrt{d}$ ($d>0$) represents the wave velocity in $\Omega$. This mathematical model is employed across diverse disciplines to represent various wave phenomena, including vibrating structures, electromagnetic waves, acoustic waves, etc. Solving the wave equation involves specifying initial and boundary conditions to accurately depict the behavior of the wave over space and time. The equations of motion for a given wave can be derived from the principle of stationary action. Classical boundary conditions (Dirichlet, Neumann, or Robin) neglect the momentum of the wave on the boundary $\Gamma:=\partial\Omega$ of the physical domain. Taking into account this momentum and the bulk-surface interaction in the medium gives rise to dynamic boundary conditions of the form $\ptt y_{\mid_{\Gamma}}-\delta\Delta_\Gamma y_{\mid_{\Gamma}}=-d\pnu y,$
where $\sqrt{\delta}$ ($\delta>0$) represents the wave velocity on $\Gamma$, $y_{\mid_{\Gamma}}$ is the trace of $y$ on $\Gamma$, $\Delta_\Gamma$ is the Laplace-Beltrami operator, and $\pnu y$ is the normal derivative of $y$ with respect to the outer unit normal vector field $\nu$. These boundary conditions are also known as nonlocally reacting kinetic conditions, as they arise from a kinetic energy function with terms on the boundary $\Gamma$ (potential energy could also be considered). The conormal derivative term $-d\pnu y$ acts as a source on $\Gamma$. For more details, we refer to \cite{H17}.

Let $T>0$ be a fixed time and $\Omega\subset \R^n$ a bounded domain ($n\in \N$) with smooth boundary $\Gamma$ such that $\Gamma=\Gamma^0\cup \Gamma^1$, with $\Gamma^0$ and $\Gamma^1$ are two closed subsets and $\Gamma^0\cap \Gamma^1=\varnothing$. We consider the wave equation with dynamic/Dirichlet (kinetic) boundary conditions
\begin{align}
	\label{intro:problem:01}
	\begin{cases}
		\ptt y-d\Delta y + q_{\Omega}(x) y=f(x,t), &\text{ in } \Omega\times(0,T),\\
		\ptt y_\Gamma -\delta \Delta_\Gamma y_\Gamma + d\pnu y + q_{\Gamma}(x) y_\Gamma=g(x,t), \quad y_\Gamma = y_{\mid_{\Gamma}}, &\text{ on }\Gamma^1\times(0,T) ,\\
		y=0,&\text{ on }\Gamma^0\times(0,T),\\
		(y(\cdot,0),y_\Gamma (\cdot,0))=(0,0), \quad(\pt y(\cdot,0),\pt y_\Gamma (\cdot,0))=(0,0), &\text{ in }\Omega\times \Gamma^1.
	\end{cases}
\end{align}
The above system models, for instance, vibrations of the membrane of a bass drum \cite{EV'2017}, where the quantity $(y(x,t),y_\Gamma(x,t))$ describes the displacements of the membrane at point $x$ and time $t$. The constants $\sqrt{d}$ and $\sqrt{\delta}$ ($d,\delta >0$) represent the wave velocities in $\Omega$ and $\Gamma$, respectively. Moreover, $f$ and $g$ designate the forcing terms. The potentials in system \eqref{intro:problem:01} are assumed to be bounded, i.e., $q_{\Omega} \in L^{\infty}(\Omega)$ and $q_{\Gamma} \in L^{\infty}\left(\Gamma^1\right)$.

We aim to establish global Lipschitz stability for the source terms in system \eqref{intro:problem:01}. Our purpose is to determine a couple of unknown forcing terms $\mathbf{F}:= (f,g)$ in \eqref{intro:problem:01} that belong to an admissible set from a single partial measurement of the boundary flux $\pnu y$ on a subset of $\Gamma^0$ (see Theorem \ref{thm:stab}).

Hyperbolic evolution equations with dynamic-type boundary conditions have attracted considerable attention in the last years and have found many applications in various situations where wave propagation occurs along the boundary of a physical domain. Examples include vibrating membranes, equations of motion, structural acoustics, and oscillatory phenomena. The study of wave equations with kinetic boundary conditions of dynamic type requires careful consideration of both the PDE describing the evolution of the system inside the bulk and the evolution on its surface. Taking into account the coupled nature of the wave system, recent research has been done for the physical derivation \cite{Go'06}, well-posedness and regularity of such equations, see \cite{BE04, FGG05, GGG03, GL14, GS16, Gu20} and \cite{KW06, Mugnolo'2011, EV'2017, EV'2018, EV'13, XL03, XL04}. However, in terms of inverse problems and controllability, the existing literature is quite limited compared to the static case (Dirichlet, Neumann, and Robin) despite their importance in applications. Interested readers can be directed to \cite{Im02, IY01, IY012, BY17} for inverse problems in the static case. Further results on controllability for wave-like systems with dynamic boundary conditions have been investigated in \cite{AL03, BBE13, CNR23, LMR22, LX16, Vanspranghe'2020, ZG19}. More recently, the authors have numerically studied an inverse source problem for a one-dimensional wave equation with dynamic boundary conditions \cite{CGMZ23}. As for the numerical analysis of wave equations with dynamic boundary conditions, we refer to the recent works \cite{H17, HK21}.

The so-called Carleman estimates provide a powerful tool for studying inverse problems and controllability. Recently, Gal and Tebou \cite{Tebou2017} have proposed a Carleman estimate in the context of boundary controllability of a nonconservative version of \eqref{intro:problem:01} at any $T>T_0$ ($T_0$ is given explicitly). The main result of the aforementioned work has been established using the classical weight function
\begin{equation}\label{weq}
    \zeta(x, t)=|x-x_0|^2-\beta t^2 + C_0, \qquad (x,t) \in \Omega\times (-T,T)
\end{equation}
for some $x_0\notin \overline{\Omega}$ and constants $\beta,  C_0>0$ ($|\cdot|$ is the Euclidean norm). The authors have used the following formula for the Hessian in $\Gamma$ (as a Riemannian manifold) of $\zeta$:
\begin{equation}\label{assu1.1}
\nabla_{\Gamma}^2 \zeta=2 \mathbb{I}_{(n-1) \times(n-1)},
\end{equation}
where $\mathbb{I}$ denotes the identity matrix, see \cite[p.~338]{Tebou2017}. It turned out that the formula \eqref{assu1.1} played a crucial role in proving the Carleman estimate in \cite[Theorem 2.2]{Tebou2017}, as it implies the pseudo-convexity on the boundary (see H\"ormander \cite[Section 8.6]{Ho63}). However, it seems that this identity does not hold even for a standard example; see Remark \ref{contr1}.

In the present paper, our initial step consists of proving a global Carleman estimate for the system \eqref{intro:problem:01} that corrects and improves the previous one from \cite{Tebou2017}. This is achieved by modifying the previous weight function with another function used for the Schr\"odinger equation. Since our Carleman estimate improves its counterpart in \cite{Tebou2017}, a sharp boundary controllability result corresponding to \eqref{intro:problem:01} can be derived following a standard duality argument. We refer to \cite{BDEM22} for an observability result for \eqref{intro:problem:01} in the case of an annulus of $\mathbb R^2$, for which we consider a more general setting and give an explicit time required for the observability. Furthermore, we shall consider a new application of our Carleman estimate to an inverse problem by proving a global Lipschitz stability for the simultaneous recovery of two forcing terms from a single boundary measurement. As for parabolic equations with dynamic boundary conditions, we refer to \cite{ACMO21, MMS'17}. To the best of the authors' knowledge, the global Lipschitz stability for the inverse hyperbolic problem with general dynamic/Dirichlet boundary conditions has not been considered in the literature.

The rest of the paper is organized as follows: Section \ref{sec2} briefly discusses the functional setting of system \eqref{intro:problem:01}. Special attention is then paid to the weight functions needed to prove the Carleman estimate in Section \ref{sec3}. Section \ref{sec4} is devoted to the Lipschitz stability for the inverse source problem. In Section \ref{sec5}, we outline the exact boundary controllability corresponding to \eqref{intro:problem:01} with general initial data. Finally, in Section \ref{sec6}, we present some open problems related to Carleman estimates for \eqref{intro:problem:01}.

\section{General setting}\label{sec2}

\subsection{Functional spaces}
We will use the following real space $\mathcal{H} := L^2(\Omega,\d x) \times L^2\left(\Gamma^1, \d S\right),$ which is a Hilbert space equipped with the inner product given by
\begin{align*}
&\langle (u,u_\Gamma),(v,v_\Gamma)\rangle_{\mathcal{H}}:=\langle (u,u_\Gamma),(v,v_\Gamma)\rangle =\langle u,v\rangle_{L^2(\Omega)} +\langle u_\Gamma,v_\Gamma\rangle_{L^2\left(\Gamma^1\right)},
\end{align*}
where the Lebesgue measure on $\Omega$ is denoted by $\d x$ and the surface measure on $\Gamma$ by $\d S$. We also consider the following spaces:
\begin{align*}
H_{0,\Gamma^0}^k(\Omega) &:=\left\{y \in H^k(\Omega): y=0 \text { on } \Gamma^0\right\},\\
\mathcal{E}^k &:=\left\{\left(y, y_{\Gamma}\right) \in H_{ 0,\Gamma^0}^k(\Omega) \times H^k\left(\Gamma^1\right): y_{\Gamma}=y_{\mid_{\Gamma^1}}\right\}, \qquad k=1,2.
\end{align*}
For simplicity, we set $\mathcal{E}^1 = \mathcal{E}$ and $\mathcal{E}^{-1}$ for its dual space with respect to the pivot $\mathcal{H}$. The well-posedness result for system \eqref{intro:problem:01} can be derived following the approach outlined in \cite[Theorem 2.1]{Tebou2017}, leveraging the theory of semigroups.

\subsection{Weight functions}
We start with the following remark that presents a counterexample to the identity \eqref{assu1.1}.
\begin{remark}\label{contr1}
The weight function defined $\zeta$ by \eqref{weq} does not satisfy the identity \eqref{assu1.1}. Indeed, for $\Omega = \{(x,y,z) \in \R^3, \, x^2 +y^2 +z^2 < 1 \}$, the boundary $\Gamma:=\mathbb{S}^2$ is given by the unit $2$-sphere. Let $x_0 = (0,0,2) \notin \overline{\Omega}$ and consider the following parameterization of the $2$-sphere 
\begin{align*}
	\Phi: [0, 2\pi) \times [0,\pi)\,\, &\mapsto \mathbb{S}^2\\
	(\theta, \varphi)\,\, &\mapsto (\cos(\theta)\sin(\varphi), \sin(\theta)\sin(\varphi), \cos(\varphi)).
	\end{align*}
We have $\zeta(x,y,z,t)=\zeta(\Phi(\theta,\varphi),t)=5-4\cos(\varphi)-\beta t^2 +C_0$ on $\Gamma$. Using formula (4.3.53) in \cite{Jo17}, a simple calculation yields
\begin{equation*}
\nabla_{\Gamma}^2 \zeta = \begin{pmatrix}
4\sin^{2}(\varphi)\cos(\varphi) &  0 & \\
 \\
0  & 4 \cos(\varphi) & 
\end{pmatrix}   \neq 2 \mathbb{I}_{2 \times 2}.
\end{equation*}
\end{remark}
To define appropriate weight functions, we consider a specific type of convex set with some smoothness property of its boundary. We refer to \cite{MM'2023} for the details.

Henceforth, we assume that $\Omega = \Omega_0 \setminus \overline \Omega_1$, where $\Omega_0$ is an open bounded set with $C^2$-boundary and $\Omega_1$ is an open strongly convex set such that $\overline{\Omega_1}\subset \Omega_0$ (see \cite[Definition 1.2]{MM'2023}). Additionally, we may assume that $0 \in \Omega_1$ (up to a translation) so that $0\notin \overline{\Omega}$. A standard example is given by the annulus defined by $\Omega = \{ x \in \R^n \, : \, r_1 < |x| < r_2 \}=B_{r_2}\setminus \overline{B_{r_1}}$, with $0< r_1 < r_2$ and $B_r$ denotes the open ball in $\mathbb{R}^n$ of center $0$ and radius $r$. Furthermore, we set $\Gamma^{k} = \partial \Omega_k$ for $k=0,1$.

Following \cite{MM'2023}, we consider the Minkowski (gauge) function of the set $\Omega_1$ defined by
\begin{equation} \label{Mink}
\mu(x) = \inf \{\lambda \, \colon \lambda > 0 \, \text{ and } \, x \in \lambda \Omega_1\}, \qquad x \in \R^n.
\end{equation}
To ensure sufficient regularity of the function $\mu$, we assume that $\Gamma^1$ is of class $C^4$. We also consider
\begin{equation} \label{psi0}
\psi_{0}(x)=\mu^2(x), \qquad x \in \R^n.
\end{equation}
We shall use the following properties of the function $\psi_0$.
\begin{proposition}\label{propw}
Under the above assumptions, the function $\psi_0$ defined by \eqref{psi0} satisfies the following properties:
\begin{itemize}
	\item[$\mathrm{(i)}$] $\psi_{0}\in C^4(\overline{\Omega})$,
	\item[$\mathrm{(ii)}$] $\psi_{0} = 1$ on $\Gamma^1$,
	\item[$\mathrm{(iii)}$] $\nabla \psi_{0} \neq 0$ in $\overline{\Omega}$,
	\item[$\mathrm{(iv)}$] There exists $\rho > 0$ such that $\nabla^2 \psi_{0} (\xi, \xi) \geq 2\rho |\xi|^2$  in $\overline{\Omega}$ for all $\xi \in \mathbb{R}^n$,
        \item[$\mathrm{(v)}$] $\pnu \psi_0 < 0$ on $\Gamma^1$.
\end{itemize}
\end{proposition}

\begin{proof}
We refer to \cite[Proposition 3.1]{MM'2023} for the proof of $\mathrm{(i)}$-$\mathrm{(iv)}$. We sketch the proof of $\mathrm{(v)}$ using the well-known properties of the function $\mu$. For any $x \in \Gamma^1$, $\partial_\nu \psi_0(x) =\lim\limits_{h \rightarrow 0^{+}} \frac{\psi_0(x + h \nu)-1}{h} \leq 0,$ because $\psi_0<1$ on $\Omega_1$ and $\psi_0(x)=1$. On the other hand, using $\mathrm{(ii)}$-$\mathrm{(iii)}$, we obtain $\left|\partial_\nu \psi_0(x)\right|=\left|\nabla \psi_0(x)\right|>0$. Thus, $\partial_\nu \psi_0(x)<0$.
\end{proof}

Then, we define the Carleman weight functions as follows
\begin{align}
\label{def:weight:functions}
\psi(x,t)=\psi_{0}(x)-\beta t^2+C_{1}\quad \text{and} \quad\varphi(x,t)=e^{\lambda \psi(x,t)}, \qquad  (x,t)\in \overline{\Omega}\times [-T,T], 
\end{align}
where $0<\beta<\rho d $, $C_{1}$ is chosen so that $\psi>1$, and $\lambda$ is a parameter (sufficiently large).

\begin{remark} 
$\bullet$ By the convexity assumption on $\Omega_1$, the weight function $\psi$ is constant with respect to $x$ on $\Gamma^1$. This implies that $\nabla_\Gamma \psi = 0$ and $\nabla_\Gamma^2 \psi=0$ on $\Gamma^1\times [-T,T]$.\\
$\bullet$ The minimal constant $\rho$ in Proposition \ref{propw}-$\mathrm{(iv)}$ is related to the smallest eigenvalue of $\nabla^2 \psi_0$, which is linked to the curvature of $\Gamma^1$ (via the second fundamental form); see \cite[\S 2.5]{Sch13} for more details.
\end{remark}
The weight function defined in \eqref{psi0} has been recently employed to prove a Carleman estimate for a Schr\"odinger equation with dynamic boundary conditions, leading to an exact controllability result in \cite{MM'2023}. We also refer to \cite{BM'2008, BMO07} for some relevant inverse problems. To the best of the authors' knowledge, the function $\psi$ has not been used in the context of wave equations with dynamic boundary conditions.

\section{Carleman estimate}\label{sec3}
This section presents our key result on the Carleman estimate for \eqref{intro:problem:01}. First, we introduce the following notations:
\begin{align}\label{Inclu}
\Omega_T = \Omega \times (-T,T), \quad &\gamma_T = \gamma \times (-T,T)\quad \text{ and } \qquad \Gamma^k_{T} = \Gamma^k \times (-T,T) \quad \text{ for } k=0,1, \notag\\
&\gamma:=\left\{ x\in\Gamma \colon \;\pnu\psi_0(x)\geq 0 \right\} \subseteq \Gamma^{0}, \qquad C^\prime=\min_{x\in\Gamma^1} |\partial_\nu \psi_0(x)|.
\end{align}
Note that $C^\prime >0$ by Proposition \ref{propw}-$\mathrm{(v)}$.

\begin{theorem}\label{Thm:Carleman}
There exist positive constants $C$, $s_1$ and $\lambda_1$ such that for all $ \lambda\geq \lambda_1$ and $ s\geq s_1$, the following Carleman estimate holds
\begin{align}\label{carleman}
	&\int_{\Omega_T} {e^{2s\varphi} \left(s^3\lambda^3\varphi^3|y|^2 +s\lambda \varphi |\nabla y|^2 +s\lambda \varphi |\pt y|^2 \right) }\mathrm{d}x \mathrm{d}t \notag\\
	&+\int_{\Gamma_{T}^{1}} {e^{2s\varphi} \left(s^3\lambda^3\varphi^3 |y_\Gamma|^2 + s\lambda \varphi|\pnu y|^2 + s\lambda \varphi|\pt y_\Gamma|^2 + [C^\prime d(\delta-d)-8\beta\delta]s\lambda \varphi |\nabla_\Gamma y_\Gamma|^2 \right)}\mathrm{d}S\mathrm{d}t \notag\\
	\leq & C\int_{\Omega_{T}}{e^{2s\varphi}|f|^2}\mathrm{d}x\mathrm{d}t
	+ C \int_{\Gamma_{T}^{1}}{e^{2s\varphi} |g|^2}\mathrm{d}S\mathrm{d}t+C s\lambda \int_{\gamma_{T}}{e^{2s\varphi} \varphi|\pnu y|^2}\mathrm{d}S\mathrm{d}t \\
	& \hspace{-0.3cm} +Cs\lambda\int_{\Omega}\mathrm{e}^{2s\varphi(x,T)}\varphi(x,T)(|\pt y(x,T)|^2+|\nabla y(x,T)|^2)\mathrm{d}x+Cs^3\lambda^3\int_{\Omega}e^{2s\varphi(x,T)}\varphi^3(x,T)|y(x,T)|^2\mathrm{d}x \notag\\
    & \hspace{-0.3cm} +Cs\lambda\int_{\Omega}e^{2s\varphi(x,-T)}\varphi(x,-T) \left(|\pt y(x,-T)|^2+|\nabla y(x,-T)|^2\right)\mathrm{d}x \notag\\
    & \hspace{-0.3cm} +C s^3\lambda^3\int_{\Omega}e^{2s\varphi(x,-T)}\varphi^3(x,-T)|y(x,-T)|^2\mathrm{d}x+Cs\lambda \int_{\Gamma^1}e^{2s\varphi(x,T)}\varphi(x,T)|\pt y_\Gamma(x,T)|^2\mathrm{d}S \notag\\
    & \hspace{-0.3cm} +Cs\lambda \int_{\Gamma^1}e^{2s\varphi(x,-T)}\varphi(x,-T)|\pt y_\Gamma(x,-T)|^2\mathrm{d}S+Cs^3\lambda^3\int_{\Gamma^1}e^{2s\varphi(x,T)}\varphi^3(x,T)|y_\Gamma(x,T)|^2\mathrm{d}S \notag\\
    & \hspace{-0.3cm} +Cs^3\lambda^3\int_{\Gamma^1}e^{2s\varphi(x,-T)}\varphi^3(x,-T)|y_\Gamma(x,-T)|^2\mathrm{d}S,\notag
\end{align}   
where
$(y,y_\Gamma)\in H^2\left(-T,T; \mathcal{H}\right)\cap L^2(-T,T;\mathcal{E})$ satisfies
$$
f :=\ptt y-d\Delta y  +q_{\Omega}y\in L^2(\Omega_T) \quad \text{and}\quad  g :=\ptt y_{\Gamma} - \delta \Delta_\Gamma y_\Gamma +d\pnu y + q_{\Gamma} y_\Gamma\in L^2\left(\Gamma_T^1\right).
$$
\end{theorem}

\begin{remark}\label{rmkbeta}
The Carleman estimate \eqref{carleman} corrects and drastically improves \cite[Theorem 2.2]{Tebou2017}. First, by the modified weight function \eqref{def:weight:functions}, we use the fact that $\nabla_\Gamma^2 \psi=0$ on $\Gamma^1$ instead of the incorrect identity \eqref{assu1.1}. Secondly, the term with $|\nabla_\Gamma y_\Gamma|^2$ can be kept only on the left-hand side of the Carleman estimate as long as $\delta > d$ by choosing $\beta < \frac{C^\prime d(\delta-d)}{8\delta}$ (in addition to $\beta <\rho d$), while in \cite[Theorem 2.2]{Tebou2017} it appears on the right-hand side. Moreover, we have been able to absorb the term with $|\pt y_\Gamma|^2$ into the left-hand side, in contrast to \cite{Tebou2017}. Finally, being different from \cite[Section 2.3]{Tebou2017}, we do not assume that $(y,y_\Gamma)$ and its derivatives are null at $t=\pm T$ to avoid cut-off functions in the applications (see Sections \ref{sec4} and \ref{sec5}).
\end{remark}
\subsection{Proof of Theorem \ref{Thm:Carleman}}
Here, we prove the Carleman estimate stated in Theorem \ref{Thm:Carleman} by adopting a refined strategy compared to the one used in \cite[Theorem 2.2]{Tebou2017}. For the reader's convenience, the proof will be detailed in several steps. To simplify, we will often write $y$ instead of $y_\Gamma$ for functions on the boundary $\Gamma$. We also neglect zeroth-order terms as they do not influence the Carleman estimate.

\noindent $\bullet$ {\bf Step 1: Change of variable.}
Set $z=e^{s\varphi}y$ and
\begin{equation*}
    Pz=e^{s\varphi} (\ptt-d\Delta)(e^{-s\varphi}z)\; \text{in}\;\Omega_T \quad \text{ and } \quad P_{\Gamma}z=e^{s\varphi}(\ptt-\delta\Delta_{\Gamma}+d\pnu)(e^{-s\varphi}z) \; \text{on} \; \Gamma_{T}^1.
    \end{equation*}
Firstly, we compute $Pz$ and $P_{\Gamma}z$. We have
\begin{align}
\label{eq:P}
Pz=P_1 z+P_2 z\quad \text{ in }\Omega_T \qquad \text{ and } \qquad P_{\Gamma}z=P_{1\Gamma} z+P_{2\Gamma} z \quad \text{ on }\Gamma_{T}^1,
\end{align}  
where the operators $P_1$ and $P_2$ are given by
\begin{align*}
\begin{array}{ll}
     & P_1 z=\ptt z-d \Delta z +s^2(|\pt \varphi|^2-d|\nabla \varphi|^2)z,\\
     & P_2 z=-2s(\pt \varphi\pt z-d \nabla\varphi\cdot\nabla z)-s(\ptt\varphi-d\Delta\varphi)z,
\end{array}
\end{align*}
and $P_{1\Gamma}$ and $P_{2\Gamma}$ are given by
\begin{align*}
\begin{array}{ll}
&P_{1\Gamma} z=\ptt z-\delta\Delta_{\Gamma}z+d\pnu z+s^2(|\pt \varphi|^2-\delta|\nabla_{\Gamma}\varphi|^2)z,\\
&P_{2\Gamma}z=-2s\left(\pt\varphi\pt z-\delta\left<\nabla_{\Gamma}\varphi,\nabla_{\Gamma}z \right>_{\Gamma} \right)-s\left( \ptt\varphi-\delta \Delta_{\Gamma}\varphi+d\pnu \varphi\right)z.
\end{array}
\end{align*}  
Now, by applying the $L^2(\Omega_T)$-norm and the $L^2(\Gamma^1_T)$-norm to \eqref{eq:P}, respectively, we obtain
\begin{align}
\label{after:conjug}
\|Pz\|_{L^2(\Omega_{T})}^2+\|P_{\Gamma}z\|_{L^2(\Gamma_{T}^1)}^2=&\|P_{1}z\|_{L^2(\Omega_{T})}^2+\|P_{2}z\|_{L^2(\Omega_{T})}^2+\|P_{1\Gamma}\|_{L^2(\Gamma_{T}^1)}^2+\|P_{2\Gamma}\|_{L^2(\Gamma_{T}^1)}^2\notag\\
&+2\left< P_1 z,P_2 z \right>_{L^2(\Omega_{T})} +2\left< P_{1\Gamma} z,P_{2\Gamma} z \right>_{L^2(\Gamma_{T}^1)}\\
=&\|e^{s\varphi}f\|_{L^2(\Omega_{T})}^2+\|e^{s\varphi}g\|_{L^2(\Gamma_{T}^1)}^2.\notag
\end{align}
The next steps of the proof are dedicated to computing the terms $\left< P_1 z,P_2 z \right>_{L^2(\Omega_{T})}$ and $\left< P_{1\Gamma} z,P_{2\Gamma} z \right>_{L^2(\Gamma_{T}^1)}$.

\noindent $\bullet$ {\bf Step 2: Bulk integrals.} By computing the term $\left< P_1 z,P_2 z \right>_{L^2(\Omega_{T})}$, we obtain
\begin{align*}
    \begin{split} 
\left< P_1 z,P_2 z\right>_{L^2(\Omega_{T})}
&=-2s\int_{\Omega_{T}}\ptt z\left( \pt \varphi\pt z-d\nabla \varphi\cdot\nabla z \right)\mathrm{d}x \mathrm{d}t-s\int_{\Omega_{T}}\ptt z z\left( \ptt\varphi-d \Delta \varphi \right)\mathrm{d}x \mathrm{d}t\\
+&2ds\int_{\Omega_{T}}\Delta z \left( \pt \varphi\pt z-d\nabla\varphi \cdot \nabla z \right)\mathrm{d}x \mathrm{d}t+ds \int_{\Omega_{T}}\Delta z z \left( \ptt \varphi-d \Delta \varphi \right)\mathrm{d}x \mathrm{d}t\\
-& 2s^3\int_{\Omega_{T}}\left( |\pt \varphi|^2-d|\nabla\varphi|^2 \right)z\left( \pt\varphi \pt z-d\nabla\varphi\cdot\nabla z\right)\mathrm{d}x\mathrm{d}t\\
-&s^3\int_{\Omega_{T}}\left( |\pt \varphi|^2-d|\nabla \varphi|^2 \right)\left( \ptt\varphi-d\Delta\varphi \right)|z|^2\mathrm{d}x \mathrm{d}t =\sum_{k=1}^6 I_{k}.
\end{split}
\end{align*}
Next, we calculate the six terms $I_{k}$, $k=1,\ldots,6,$ by using integration by parts in space (Green's formula) and in time. Now, $I_{1}$ is given by
\begin{align*}
I_{1}=& -s\left[ \int_{\Omega} \pt\varphi |\pt z|^2 \mathrm{d}x\right]_{-T}^{T}+s\int_{\Omega_{T}}|\pt z|^2\left( \ptt\varphi+d\Delta \varphi\right)\mathrm{d}x\mathrm{d}t-2ds\int_{\Omega_{T}}\pt z \nabla \pt\varphi\cdot \nabla z \mathrm{d}x\mathrm{d}t\\
&-ds\int_{\Gamma_{T}}\partial_{\nu}\varphi |\pt z|^2\mathrm{d}S\mathrm{d}t+2sd\left[ \int_{\Omega}\pt z \nabla\varphi\cdot\nabla z\mathrm{d}x \right]_{-T}^{T}.
\end{align*}    
Integration by parts yields
\begin{align*}
I_{2}=&-s\left[ \int_{\Omega}\pt z z\left( \pt^2\varphi-d\Delta\varphi \right) \mathrm{d}x\right]_{-T}^T+s\int _{\Omega_{T}} |\pt z|^2 \left( \ptt \varphi-d\Delta\varphi \right)\mathrm{d}x\mathrm{d}t-\dfrac{s}{2}\int_{\Omega_{T}}|z|^2\left( \ptt-d \Delta \right)\ptt\varphi \mathrm{d}x\mathrm{d}t\\
&+\dfrac{s}{2}\left[ \int_{\Omega}|z|^2(\pt^2-d\Delta)\pt\varphi \mathrm{d}x\right]_{-T}^T.
\end{align*}
Furthermore, by Green’s formula and integration by parts, we obtain
\begin{align*}
        I_{3}=&2ds\int_{\Gamma_{T}}\left[ \pnu z(\pt\varphi\pt z-d\nabla \varphi\cdot \nabla z)+\dfrac{d}{2}\pnu \varphi |\nabla z|^2 \right]\mathrm{d}S\mathrm{d}t-ds\left[\int_{\Omega}|\nabla z|^2\pt\varphi\mathrm{d}x\right]_{-T}^{T}\\
        &+ds\int_{\Omega_T}\left[ |\nabla z|^2(\ptt\varphi-d\Delta\varphi)-2\pt z\nabla z\cdot\nabla\pt\varphi +2d\nabla^2\varphi(\nabla z,\nabla z)\right]\mathrm{d}x\mathrm{d}t.
\end{align*}
On the other hand
\begin{align*}
        I_{4}=&-ds\int_{\Omega_T} \left[  |\nabla z|^2\left( \ptt \varphi-d\Delta\varphi \right)-\dfrac{1}{2}|z|^2\Delta(\ptt\varphi-d\Delta\varphi)  \right]\mathrm{d}x\mathrm{d}t\\
        &+ds\int_{\Gamma_T}\left[\pnu z z (\ptt\varphi-d\Delta\varphi)-\dfrac{1}{2}|z|^2\pnu(\ptt\varphi-d\Delta\varphi)\right]\mathrm{d}S\mathrm{d}t.
\end{align*}
Next, we have
\begin{align*}
        I_{5}=&s^3\int_{\Omega_T}|z|^2(\ptt \varphi-d\Delta\varphi)(|\pt\varphi|^2-d|\nabla\varphi|^2)\mathrm{d}x\mathrm{d}t\\
        &+2s^3\int_{\Omega_T}|z|^2\left(|\pt\varphi|^2\ptt\varphi-2d\pt\varphi\nabla\varphi\cdot\nabla\pt\varphi+d^2\nabla^2\varphi(\nabla\varphi,\nabla\varphi)\right)\mathrm{d}x\mathrm{d}t\\
        &-s^3\left[\int_{\Omega} |z|^2 \pt\varphi(|\pt\varphi|^2-d|\nabla \varphi|^2)\mathrm{d}x \right]_{-T}^T+d s^3 \int_{\Gamma_T}|z|^2\pnu \varphi(|\pt\varphi|^2-d|\nabla\varphi|^2)\mathrm{d}S\mathrm{d}t.
\end{align*}
Finally,
$\displaystyle
I_6=-s^3\int_{\Omega_T}|z|^2(\ptt\varphi-d\Delta\varphi)(|\pt\varphi|^2-d|\nabla\varphi|^2)\mathrm{d}x\mathrm{d}t.
$
Elementary computations show that
\begin{align}
        &\left\langle P_{1}z,P_{2}z\right\rangle_{L^2(\Omega_T)}
        =2 s \int_{\Omega_T}\left(\partial_t^2 \varphi\left|\partial_t z\right|^2-2 d \partial_t z \nabla \partial_t \varphi \cdot \nabla z+d^2 \nabla^2 \varphi(\nabla z, \nabla z)\right) \mathrm{d} x \mathrm{d} t \notag\\ 
        &\hspace{-0.2cm} +2 s^3 \int_{\Omega_T}|z|^2\left(\left|\partial_t \varphi\right|^2 \partial_t^2 \varphi+d^2 \nabla^2 \varphi(\nabla \varphi, \nabla \varphi)-2 d \partial_t \varphi \nabla \varphi \cdot \nabla \partial_t \varphi\right) \mathrm{d} x \mathrm{~d} t-\frac{s}{2} \int_{\Omega_T}|z|^2\left(\partial_t^2-d \Delta\right)^2 \varphi \mathrm{d} x \mathrm{d} t\notag\\
        &\hspace{-0.2cm} -s\left[ \int_{\Omega}\pt\varphi|\pt z|^2 \mathrm{d} x\right]_{-T}^T+2sd\left[ \int_{\Omega}\pt z \nabla\varphi\cdot\nabla z\mathrm{d}x \right]_{-T}^{T}-s\left[ \int_{\Omega}\pt z z\left( \pt^2\varphi-d\Delta\varphi \right) \mathrm{d}x\right]_{-T}^T\notag\\
        &\hspace{-0.2cm} -ds\left[\int_{\Omega}|\nabla z|^2\pt\varphi\mathrm{d}x\right]_{-T}^{T}-s^3\left[\int_{\Omega} |z|^2 \pt\varphi(|\pt\varphi|^2-d|\nabla \varphi|^2)\mathrm{d}x \right]_{-T}^T+\dfrac{s}{2}\left[ \int_{\Omega}|z|^2(\pt^2-d\Delta)\pt\varphi \mathrm{d}x\right]_{-T}^T\notag\\
        &\hspace{-0.2cm} +sd\int_{\Gamma_T}\left[ d\pnu\varphi|\nabla z|^2-2d\pnu z\nabla\varphi\cdot \nabla z+2\pt\varphi\pt z\pnu z-\pnu \varphi |\pt z|^2 \right]\mathrm{d}S\mathrm{d}t\notag\\
        & \hspace{-0.2cm}+ds\int_{\Gamma_{T}}\left[ z\pnu z (\ptt \varphi-d \Delta \varphi)+s^2\pnu \varphi |z|^2(|\pt\varphi|^2-d |\nabla\varphi|^2)-\dfrac{1}{2}|z|^2\pnu (\ptt\varphi-d \Delta\varphi) \right]\mathrm{d}S\mathrm{d}t. \label{C1}
\end{align}
The terms on $\Omega_T$ are quite classical, similar to the Dirichlet case. However, the dynamic boundary condition does not allow the same simplification of boundary terms. Moreover, we retain the boundary terms in time to avoid the cut-off argument in applications.

\noindent $\bullet$ {\bf Step 3: Boundary terms. } We have
\begin{align*}
        &\left\langle P_{1\Gamma} z ,P_{2\Gamma} z\right\rangle_{L^2(\Gamma_{T}^1)}
        =-2s\int_{\Gamma_{T}^1}\ptt z\left( \pt\varphi\pt z-\delta\left\langle \nabla_{\Gamma}\varphi,\nabla_{\Gamma}z \right\rangle_{\Gamma} \right)\mathrm{d}S\mathrm{d}t\\
        &-s\int_{\Gamma_{T}^1}\ptt z\left(\ptt\varphi-\delta \Delta_{\Gamma}\varphi+d\pnu\varphi\right)z\mathrm{d}S\mathrm{d}t-2ds\int_{\Gamma_{T}^1}\pnu z\left( \pt\varphi\pt z-\delta\left\langle \nabla_{\Gamma}\varphi,\nabla_{\Gamma}z \right\rangle_{\Gamma} \right)\mathrm{d}S\mathrm{d}t\\
        &-ds\int_{\Gamma_{T}^1}\pnu z\left(\ptt\varphi-\delta \Delta_{\Gamma}\varphi+d\pnu\varphi  \right)z\mathrm{d}S\mathrm{d}t+2\delta s\int_{\Gamma_{T}^1}\Delta_\Gamma z \left( \pt\varphi\pt z-\delta\left\langle \nabla_{\Gamma}\varphi,\nabla_{\Gamma}z \right\rangle_{\Gamma} \right)\mathrm{d}S\mathrm{d}t\\
        &+\delta s\int_{\Gamma_{T}^1}\Delta_{\Gamma}z\left(\ptt\varphi-\delta \Delta_{\Gamma}\varphi+d\pnu\varphi  \right)z\mathrm{d}S\mathrm{d}t\\
        &-2s^3 \int_{\Gamma_{T}^1}\left( |\pt\varphi|^2-\delta|\nabla_{\Gamma}\varphi|^2 \right)z\left( \pt\varphi\pt z-\delta\left\langle \nabla_{\Gamma}\varphi,\nabla_{\Gamma}z \right\rangle_{\Gamma} \right)\mathrm{d}S\mathrm{d}t\\
        &-s^3\int_{\Gamma_{T}^1} \left( |\pt\varphi|^2-\delta|\nabla_{\Gamma}\varphi|^2 \right)\left(\ptt\varphi-\delta \Delta_{\Gamma}\varphi+d\pnu\varphi  \right)|z|^2\mathrm{d}S\mathrm{d}t =:\sum_{k=1}^8 I_{k\Gamma}.
\end{align*}
Now we calculate the eight terms $I_{k\Gamma}$, $k=1,\ldots, 8,$ as above
\begin{align*}
        I_{1\Gamma}=&\left[ -s\int_{\Gamma^1}|\pt z|^2 \pt\varphi \mathrm{d}S\right]_{-T}^{T}+s\int_{\Gamma_{T}^1}|\pt z|^2(\ptt\varphi+\delta\Delta_{
        \Gamma}\varphi)\mathrm{d}S\mathrm{d}t-2s\delta\int_{\Gamma_{T}^1}\pt z\left\langle \nabla _{\Gamma}\pt \varphi,\nabla_{\Gamma}z \right\rangle_{\Gamma}\mathrm{d}S\mathrm{d}t\\
        &+2s\delta\left[\int_{\Gamma^1}\pt z \left\langle  \nabla_{\Gamma}\varphi,\nabla_{\Gamma}z\right\rangle_{\Gamma} \mathrm{d}S \right]_{-T}^T.
\end{align*}
Similarly, we have
\begin{align*}
        I_{2\Gamma}=&-s\left[\int_{\Gamma^1} \pt z z  (\ptt \varphi-\delta \Delta_{\Gamma}\varphi+d\pnu\varphi)\mathrm{d}S\right]_{-T}^T+s\int_{\Gamma_{T}^1}|\pt z|^2 (\ptt \varphi+d\pnu\varphi-\delta\Delta_{\Gamma}\varphi)\mathrm{d}S\mathrm{d}t\\
        &-\dfrac{s}{2}\int_{\Gamma_{T}^1} |z|^2 (\ptt+d\pnu-\delta \Delta _{\Gamma})\ptt\varphi\mathrm{d}S\mathrm{d}t+\dfrac{s}{2}\left[\int_{\Gamma^1} |z|^2(\ptt+d\pnu-\delta \Delta _{\Gamma})\pt\varphi\mathrm{d}S \right]_{-T}^T.
\end{align*}
On the other hand, we have
\begin{align*}
        I_{3\Gamma}=-2ds\int_{\Gamma_{T}^1}\pnu z \left(\pt\varphi\pt z-\delta\left\langle\nabla_{\Gamma}\varphi,\nabla_{\Gamma}z  \right\rangle_{\Gamma}  \right)\mathrm{d}S\mathrm{d}t \text{ and } I_{4\Gamma}=-ds\int_{\Gamma_{T}^1}\pnu z z \left( \ptt\varphi+d\pnu\varphi-\delta\Delta_{\Gamma}\varphi \right)\mathrm{d}S\mathrm{d}t.
\end{align*}
Next, by using the surface Green formula, we obtain
\begin{align*}
        I_{5\Gamma}=&\delta s\int_{\Gamma_{T}^1} \left[|\nabla_{\Gamma}z|^2(\ptt\varphi-\delta\Delta_{\Gamma}\varphi) -2\left\langle \nabla_{\Gamma}z,\nabla_{\Gamma}\pt\varphi \right\rangle_{\Gamma}\pt z +2\delta\nabla_{\Gamma}^2\varphi(\nabla_{\Gamma}z,\nabla_{\Gamma} z) \right]\mathrm{d}S\mathrm{d}t
        -\delta s\left[\int_{\Gamma^1}|\nabla_{\Gamma}z|^2\pt\varphi \mathrm{d}S\right]_{-T}^T.
\end{align*}
In the same manner, $I_{6\Gamma}$ can be computed as follows
\begin{align*}
    \begin{split}
        I_{6\Gamma}=&-\delta s \int_{\Gamma_{T}^1}\left[ |\nabla_{\Gamma}z|^2\left( \ptt\varphi+d\pnu\varphi-\delta\Delta_{\Gamma}\varphi \right)+\dfrac{1}{2}|z|^2\Delta_{\Gamma}\left( \ptt\varphi+d\pnu\varphi-\delta\Delta_{\Gamma}\varphi \right)   \right]\mathrm{d}S\mathrm{d}t.
    \end{split}
\end{align*}
We integrate by parts in time and use the surface Green formula to obtain
\begin{align*}
    \begin{split}
        I_{7\Gamma}=&-s^3\left[ \int_{\Gamma^1}|z|^2\pt\varphi(|\pt\varphi|^2-\delta|\nabla_{\Gamma}\varphi|^2)\mathrm{d}S
 \right]_{0}^T+s^3\int_{\Gamma_{T}^1}|z|^2\left(\ptt\varphi-\delta\Delta_{\Gamma}\varphi  \right)\left( |\pt \varphi|^2-\delta |\nabla_{\Gamma}\varphi|^2 \right)\mathrm{d}S\mathrm{d}t\\
        &+2s^3\int_{\Gamma_{T}^1}|z|^2\left( |\pt\varphi|^2\ptt\varphi-2\delta \pt\varphi \left\langle \nabla_{\Gamma}\varphi,\nabla_{\Gamma}\pt\varphi  \right\rangle_{\Gamma}+\delta^2 \nabla_{\Gamma}^2 \varphi(\nabla_{\Gamma}\varphi,\nabla_{\Gamma}\varphi) \right)\mathrm{d}S\mathrm{d}t.
    \end{split}
\end{align*}
Finally, we have    
$ \displaystyle
    I_{8\Gamma}=-s^3\int_{\Gamma_{T}^1}|z|^2\left( \ptt\varphi-\delta \Delta_{\Gamma}\varphi+d\pnu\varphi \right)\left( |\pt \varphi|^2-\delta |\nabla_{\Gamma}\varphi|^2 \right)\mathrm{d}S\mathrm{d}t.
$
By collecting all the above terms, we obtain
\begin{align*}
       & \left\langle P_{1}z,P_{2}z\right\rangle_{L^2(\Omega_T)}+\left\langle P_{1\Gamma}z,P_{2\Gamma}z \right\rangle_{L^2(\Gamma_{T}^1)}
       =2 s \int_{\Omega_T}\left(\partial_t^2 \varphi\left|\partial_t z\right|^2-2 d \partial_t z \nabla \partial_t \varphi \cdot \nabla z+d^2 \nabla^2 \varphi(\nabla z, \nabla z)\right) \mathrm{d} x \mathrm{d} t  \notag\\
       & +2 s^3 \int_{\Omega_T}|z|^2\left(\left|\partial_t \varphi\right|^2 \partial_t^2 \varphi+d^2 \nabla^2 \varphi(\nabla \varphi, \nabla \varphi)-2 d \partial_t \varphi \nabla \varphi \cdot \nabla \partial_t \varphi\right) \mathrm{d} x \mathrm{~d} t-\frac{s}{2} \int_{\Omega_T}|z|^2\left(\partial_t^2-d \Delta\right)^2 \varphi \mathrm{d} x \mathrm{d} t\\
       &+2s\int_{\Gamma_{T}^1}\left( \ptt\varphi|\pt z|^2-2\delta\pt z\left\langle \nabla_{\Gamma}\pt\varphi,\nabla_{\Gamma}z \right\rangle_{\Gamma}+\delta^2\nabla_{\Gamma}^2\varphi(\nabla_{\Gamma}z,\nabla_{\Gamma}z) \right)\mathrm{d}S\mathrm{d}t \notag\\
        &+2s^3\int_{\Gamma_{T}^1}|z|^2\left( |\pt\varphi|^2\ptt\varphi+\delta^2\nabla_{\Gamma}^2\varphi(\nabla_{\Gamma}\varphi,\nabla_{\Gamma}\varphi)-2\delta \pt\varphi\left\langle \nabla_{\Gamma}\varphi,\nabla_{\Gamma}\pt\varphi \right\rangle_{\Gamma} \right)\mathrm{d}S\mathrm{d}t\\
        &-\dfrac{s}{2}\int_{\Gamma_{T}^1}|z|^2(\ptt-\delta\Delta_{\Gamma})^2\varphi\mathrm{d}S\mathrm{d}t-s\left[ \int_{\Omega}\pt\varphi|\pt z|^2 \right]_{-T}^T-\left[ s\int_{\Gamma^1}|\pt z|^2\pt\varphi\mathrm{d}S \right]_{-T}^T \notag\\
        &+2sd\left[ \int_{\Omega}\pt z \nabla\varphi\cdot\nabla z\mathrm{d}x \right]_{-T}^{T}-s\left[ \int_{\Omega}\pt z z\left( \pt^2\varphi-d\Delta\varphi \right) \mathrm{d}x\right]_{-T}^T-ds\left[\int_{\Omega}|\nabla z|^2\pt\varphi\mathrm{d}x\right]_{-T}^{T} \notag\\
        &-s^3\left[\int_{\Omega} |z|^2 \pt\varphi(|\pt\varphi|^2-d|\nabla \varphi|^2)\mathrm{d}x \right]_{-T}^T+\dfrac{s}{2}\left[ \int_{\Omega}|z|^2(\pt^2-d\Delta)\pt\varphi \mathrm{d}x\right]_{-T}^T \notag\\
         &+2s\delta\left[\int_{\Gamma^1}\pt z \left\langle  \nabla_{\Gamma}\varphi,\nabla_{\Gamma}z\right\rangle_{\Gamma} \mathrm{d}S \right]_{-T}^T-s\left[\int_{\Gamma^1} \pt z z  (\ptt \varphi-\delta \Delta_{\Gamma}\varphi+d\pnu\varphi)\mathrm{d}S\right]_{-T}^T \notag\\
        &+\dfrac{s}{2}\left[\int_{\Gamma^1} |z|^2(\ptt+d\pnu-\delta \Delta _{\Gamma})\pt\varphi\mathrm{d}S \right]_{-T}^T-\delta s\left[\int_{\Gamma^1}|\nabla_{\Gamma}z|^2\pt\varphi \mathrm{d}S\right]_{-T}^T  -s^3\left[ \int_{\Gamma^1}|z|^2\pt\varphi(|\pt\varphi|^2-\delta|\nabla_{\Gamma}\varphi|^2)\mathrm{d}S
 \right]_{-T}^T  \notag\\
        &+sd\int_{\Gamma_T}\left[ d\pnu\varphi|\nabla z|^2-2d\pnu z\nabla\varphi\cdot \nabla z+2\pt\varphi\pt z\pnu z-\pnu \varphi |\pt z|^2 \right]\mathrm{d}S\mathrm{d}t \notag\\
        &-ds\int_{\Gamma_{T}^1}\left[ \delta\pnu\varphi|\nabla_{\Gamma}z|^2-2\delta\pnu z\left\langle \nabla_{\Gamma}\varphi,\nabla_{\Gamma}z \right\rangle_{\Gamma}+ 2\pt\varphi\pt z\pnu z-\pnu \varphi|\pt z|^2  \right]\mathrm{d}S\mathrm{d}t \notag\\
         &+ds\int_{\Gamma_{T}}\left[ z\pnu z (\ptt \varphi-d \Delta \varphi)+s^2\pnu \varphi |z|^2(|\pt\varphi|^2-d |\nabla\varphi|^2)-\dfrac{1}{2}|z|^2\pnu (\ptt\varphi-d \Delta\varphi) \right]\mathrm{d}S\mathrm{d}t \notag\\
         &-ds \int_{\Gamma_{T}^1}\left[ z \pnu z\left( \ptt\varphi+d\pnu \varphi-\delta\Delta_{\Gamma}\varphi \right)+s^2\pnu\varphi|z|^2\left( |\pt\varphi|^2-\delta|\nabla_{\Gamma}\varphi|^2 \right)+\dfrac{1}{2}|z|^2\pnu\left( \ptt\varphi-\delta\Delta_{\Gamma}\varphi \right) \right]\mathrm{d}S\mathrm{d}t.
\end{align*}
Using the fact that $|\nabla z|^2=|\pnu z|^2+|\nabla_{\Gamma}z|^2$, $\nabla\varphi\cdot\nabla z=\pnu\varphi\pnu z+\left\langle \nabla_{\Gamma}\varphi,\nabla_{\Gamma}z \right\rangle_{\Gamma}$ on $\Gamma$, and $z=0$ on $\Gamma_{T}^0$, we obtain
\begin{align*}
    \begin{split}
        & sd\int_{\Gamma_T}\left[ d\pnu\varphi|\nabla z|^2-2d\pnu z\nabla\varphi\cdot \nabla z+2\pt\varphi\pt z\pnu z-\pnu \varphi |\pt z|^2 \right]\mathrm{d}S\mathrm{d}t\\
        &-ds\int_{\Gamma_{T}^1}\left[ \delta\pnu\varphi|\nabla_{\Gamma}z|^2-2\delta\pnu z\left\langle \nabla_{\Gamma}\varphi,\nabla_{\Gamma}z \right\rangle_{\Gamma}+ 2\pt\varphi\pt z\pnu z-\pnu \varphi|\pt z|^2  \right]\mathrm{d}S\mathrm{d}t\\
        =&-sd^2\int_{\Gamma_{T}}\pnu\varphi|\pnu z|^2\mathrm{d}S\mathrm{d}t+sd(d-\delta)\int_{\Gamma_{T}^1}\left( \pnu\varphi |\nabla_{\Gamma}z|^2-2\pnu z\left\langle \nabla_{\Gamma}\varphi,\nabla_{\Gamma}z  \right\rangle_{\Gamma} \right)\mathrm{d}S\mathrm{d}t.
    \end{split}
\end{align*}
Although $\nabla_{\Gamma}\varphi=0$ on $\Gamma^1$, we retain the corresponding terms to maintain a general computation applicable to similar cases.

Therefore, $\left\langle P_{1}z,P_{2}z\right\rangle_{L^2(\Omega_T)}+\left\langle P_{1\Gamma}z,P_{2\Gamma}z \right\rangle_{L^2(\Gamma_{T}^1)}=:\sum\limits_{k=1}^3 (J_k+J_{k\Gamma})+B_{0},$ where
\begin{align}
       J_1 =&2 s \int_{\Omega_T}\left(\partial_t^2 \varphi\left|\partial_t z\right|^2-2 d \partial_t z \nabla \partial_t \varphi \cdot \nabla z+d^2 \nabla^2 \varphi(\nabla z, \nabla z)\right) \mathrm{d} x \mathrm{d} t,\notag\\
       J_2 =&2 s^3 \int_{\Omega_T}|z|^2\left(\left|\partial_t \varphi\right|^2 \partial_t^2 \varphi+d^2 \nabla^2 \varphi(\nabla \varphi, \nabla \varphi)-2 d \partial_t \varphi \nabla \varphi \cdot \nabla \partial_t \varphi\right) \mathrm{d} x \mathrm{d} t,\notag\\
       J_3 =&-\frac{s}{2} \int_{\Omega_T}|z|^2\left(\partial_t^2-d \Delta\right)^2 \varphi\, \mathrm{d} x \mathrm{d} t,\notag\\
       J_{1\Gamma} =&2s\int_{\Gamma_{T}^1}\left( \ptt\varphi|\pt z|^2-2\delta\pt z\left\langle \nabla_{\Gamma}\pt\varphi,\nabla_{\Gamma}z \right\rangle_{\Gamma}+\delta^2\nabla_{\Gamma}^2\varphi(\nabla_{\Gamma}z,\nabla_{\Gamma}z) \right)\mathrm{d}S\mathrm{d}t,\notag\\
        J_{2\Gamma} =&2s^3\int_{\Gamma_{T}^1}|z|^2\left( |\pt\varphi|^2\ptt\varphi+\delta^2\nabla_{\Gamma}^2\varphi(\nabla_{\Gamma}\varphi,\nabla_{\Gamma}\varphi)-2\delta \pt\varphi\left\langle \nabla_{\Gamma}\varphi,\nabla_{\Gamma}\pt\varphi \right\rangle_{\Gamma} \right)\mathrm{d}S\mathrm{d}t,\notag\\
        J_{3\Gamma} =&-\dfrac{s}{2}\int_{\Gamma_{T}^1}|z|^2(\ptt-\delta\Delta_{\Gamma})^2\varphi\,\mathrm{d}S\mathrm{d}t, \text{ and}\notag\\
        B_0=&-s\lambda\left[ \int_{\Omega}\pt\psi\varphi|\pt z|^2 \mathrm{d}x\right]_{-T}^T-s\lambda\left[ \int_{\Gamma^1}\pt\psi \varphi|\pt z|^2\mathrm{d}S \right]_{-T}^T-d^2s\lambda\int_{\Gamma_{T}}\pnu\psi|\pnu z|^2\mathrm{d}S\mathrm{d}t\notag\\
           &+2sd\lambda\left[ \int_{\Omega}\varphi\pt z \nabla\psi\cdot\nabla z\mathrm{d}x \right]_{-T}^T -s\lambda\left[ \int_{\Omega}\varphi \pt z z (\pt^2 \psi-d\Delta\psi)\mathrm{d}x \right]_{-T}^T\notag \\
           &-s\lambda^2\left[ \int_{\Omega}\varphi \pt z z(|\pt \psi|^2-d|\nabla\psi|^2)\mathrm{d}x \right]_{-T}^T\notag\\
       &-ds\lambda\left[ \int_{\Omega}\varphi \pt\psi |\nabla z|^2 \mathrm{d}x\right]_{-T}^T-s^3\lambda^3\left[ \int_{\Omega}\varphi^3|z|^2\pt\psi (|\pt\psi|^2-d|\nabla\psi|^2)\mathrm{d}x \right]_{-T}^T\notag\\
       &+\dfrac{s}{2}\lambda^2\left[ \int_{\Omega}\varphi |z|^2\pt\psi (\pt^2\psi+\lambda |\pt\psi|^2) \mathrm{d}x\right]_{-T}^T+\dfrac{s}{2}\lambda\left[ \int_{\Omega}\varphi |z|^2(\pt^3\psi+2\lambda\pt\psi\pt^2\psi) \mathrm{d}x\right]_{-T}^T\notag\\
       &-s\lambda \left[ \int_{\Gamma^1}\varphi \pt z z (\pt^2\psi+d\pnu\psi)\mathrm{d}S \right]_{-T}^T-s\lambda^2\left[ \int_{\Gamma^1} \varphi \pt z z|\pt\psi|^2 \mathrm{d}S\right]_{-T}^T-s^3\lambda^3\left[ \int_{\Gamma^1}\varphi^3 |z|^2 (\pt\psi)^3 \mathrm{d}S\right]_{-T}^T \notag\\
       &+\dfrac{s}{2}\lambda^2\left[ \int_{\Gamma^1}\varphi |z|^2 \pt\psi(3\pt^2\psi+\lambda|\pt\psi|^2)\mathrm{d}S \right]_{-T}^T-\delta s \lambda \left[ \int_{\Gamma^1}\varphi |\nabla_{\Gamma}z|^2\pt\psi\mathrm{d}S \right]_{-T}^T \notag\\
       &+d(d-\delta)s\lambda\int_{\Gamma_{T}^1}\pnu\psi\varphi|\nabla_\Gamma z|^2\mathrm{d}S\mathrm{d}t -ds\lambda\int_{\Gamma_T ^1}z\pnu z\varphi \left( d\pnu\psi+d\Delta\psi+\lambda d|\nabla\psi|^2 \right)\mathrm{d}S\mathrm{d}t\notag\\
       &-d^2s^3\lambda^3\int_{\Gamma_{T}^1}|z|^2\varphi\pnu\psi|\nabla\psi|^2\mathrm{d}S\mathrm{d}t \notag\\
        &-ds\lambda^2\int_{\Gamma_{T}^1}\varphi|z|^2\pnu\psi\left( \ptt\psi+\lambda|\pt\psi|^2-\dfrac{d}{2}\Delta\psi-\dfrac{d\lambda}{2}|\nabla\psi|^2 \right)\mathrm{d}S\mathrm{d}t \notag\\
         &+ds\lambda\int_{\Gamma_{T}^1}\varphi |z|^2\left( \dfrac{d}{2}\pnu \Delta\psi+\dfrac{d\lambda}{2}\pnu(|\nabla\psi|^2)\right)\mathrm{d}S\mathrm{d}t, \label{B0}
\end{align}
where we have used
\begin{align}
\label{estimate:der:weights:01}
\begin{split}
    &\pt\varphi=\lambda\pt\psi\varphi,\quad \nabla \varphi=\lambda\varphi\nabla\psi, \quad \ptt\varphi=\lambda\varphi\left( \ptt\psi+\lambda|\pt\psi|^2 \right), \nabla\pt\varphi=\lambda^2\pt\psi\varphi\nabla\psi,\\
&\Delta\varphi=\lambda\varphi\left( \Delta\psi+\lambda|\nabla\psi|^2 \right),\quad\nabla^2\varphi(\nabla z,\nabla z)=\lambda\varphi\left( \nabla^2\psi(\nabla z,\nabla z)+\lambda|\nabla z\cdot\nabla\psi|^2 \right),
\end{split}
\end{align}
and 
$
\pnu \varphi=\lambda \varphi \pnu \psi \text{ on }\Gamma_T \text{ and } \nabla_{\Gamma}\varphi=\nabla_\Gamma \psi=0  \text{ on }\Gamma^1_T.
$
Next, we estimate the sum $\sum\limits_{k=1}^3 (J_k+J_{k\Gamma})$. $C$ will denote a positive generic constant. It may change from one line to another, and it is independent of the parameters $\lambda>0$ and $s>0$. We set
\begin{equation}\label{matcalJ}
    \mathcal{J} :=\sum\limits_{k=1}^3 (J_k+J_{k\Gamma})+(d\rho+\beta)\lambda\int_{\Gamma_{T}^1}\varphi\pnu\psi|z|^2\mathrm{d}S\mathrm{d}t+2(3\beta-d\rho)\int_{\Gamma_{T}^1}\varphi z\pnu z \mathrm{d}S\mathrm{d}t.
\end{equation}
By making use of \eqref{estimate:der:weights:01}, we next prove the following estimate for $\mathcal{J}$:
\begin{align}\label{Res-3.1}
\begin{split}
        & \mathcal{J} \geq C s\lambda \int_{\Omega_T} \varphi\left( |\pt z|^2+|\nabla z|^2+\varphi^2s^2\lambda^2|z|^2 \right)\mathrm{d}x\mathrm{d}t+C s \lambda\int_{\Gamma_{T}^1}\varphi |\pt z|^2\mathrm{d}S\mathrm{d}t\\
        &-8\beta s\lambda \delta\int_{\Gamma_{T}^1}\varphi|\nabla_{\Gamma} z|^2\mathrm{d}S\mathrm{d}t-\dfrac{1}{2}\|P_1 z\|_{L^2(\Omega_{T})}^2-\dfrac{1}{2}\|P_{1\Gamma} z\|_{L^2(\Gamma_{T}^1)}^2-Cs^2\lambda^2\int_{\Gamma_{T}^1}\varphi^2|z|^2\mathrm{d}S\mathrm{d}t\\
        &+s^3\lambda^3\int_{\Gamma_{T}^1}\varphi^3|\pt\psi|^2(2\lambda |\pt \psi|^2-16\beta)|z|^2\mathrm{d}S\mathrm{d}t-Cs\lambda^2\int_{\Omega}\varphi(\cdot,T)|z(\cdot,T)|^2\mathrm{d}x\\
        &-C s\lambda^2\int_{\Omega}\varphi(\cdot,-T)|z(\cdot,-T)|^2\mathrm{d}x-Cs\lambda\int_{\Omega}\varphi(\cdot,T)|\pt z(\cdot,T)|^2\mathrm{d}x\\
 &-Cs\lambda\int_{\Omega}\varphi(\cdot,-T)|\pt z(\cdot,-T)|^2\mathrm{d}x-C_Ts\lambda^2\int_{\Gamma^1}\varphi(\cdot,T)|z(\cdot,T)|^2\mathrm{d}S\\
 &-C_T s\lambda^2\int_{\Gamma^1}\varphi(\cdot,-T)|z(\cdot,-T)|^2\mathrm{d}S-Cs\lambda\int_{\Gamma^1}\varphi(\cdot,T)|\pt z(\cdot,T)|^2\mathrm{d}S\\
 &-Cs\lambda\int_{\Gamma^1}\varphi(\cdot,-T)|\pt z(\cdot,-T)|^2\mathrm{d}S.
 \end{split}
\end{align}
This will be done in several steps. Let $\varepsilon>0$. First, we claim that
\begin{align}\label{Res0}
&J_{1}+2(d\rho+\beta)ds\lambda\int_{\Gamma_{T}^1}\varphi\left( \dfrac{\lambda}{2}\pnu\psi|z|^2-z\pnu z \right)\mathrm{d}S\mathrm{d}t \notag\\
&\qquad \geq2(\rho d-\beta)s\lambda\int_{\Omega_{T}}\varphi\left( |\pt z|^2+d|\nabla z|^2 \right)\mathrm{d}x\mathrm{d}t-\left( 2(\rho d-\beta)+4\beta \right)\left[ Cs^3\lambda^3\int_{\Omega_T}\varphi^3|b(\psi)||z|^2\mathrm{d}x\mathrm{d}t\right. \notag\\
&\left.+C_{\varepsilon}s^2\lambda^2\int_{\Omega_T}\varphi^2|z|^2\mathrm{d}x\mathrm{d}t+\varepsilon\|P_{1}z\|^2 \right]
-\left( 2(\rho d-\beta)+4\beta \right)  \left|s\lambda \left[ \int_{\Omega} \varphi \pt z z \mathrm{d}x \right]_{-T}^T-\dfrac{s\lambda}{2}\left[ \int_{\Omega}\pt\varphi|z|^2\mathrm{d}x
 \right]_{-T}^T  \right|,
\end{align}
where
\begin{equation}\label{bpsi}
b(\psi)=|\pt \psi|^2-d|\nabla\psi|^2.
\end{equation}
Let us prove \eqref{Res0}. We have
\begin{align*}
    \begin{split}
        J_{1}=&2s\lambda\int_{\Omega_T}\varphi\left[ \ptt\psi|\pt z|^2+d^2\nabla^2\psi(\nabla z,\nabla z)+\lambda\left( \pt\psi\pt z-d\nabla z\cdot\nabla \psi \right)^2 \right]\mathrm{d}x\mathrm{d}t.
    \end{split}
\end{align*}
Therefore
\begin{align}\label{Res1}
        J_1 \geq 4d^2 \rho s \lambda \int_{\Omega_T}\varphi |\nabla z|^2\mathrm{d}x\mathrm{d}t-4\beta s\lambda\int_{\Omega_T}\varphi|\pt z|^2\mathrm{d}x\mathrm{d}t.
\end{align}
On the other hand, to absorb the term with $|\pt z|^2$, we compute
\begin{align*}
&\left\langle P_1 z,s\lambda \varphi z \right\rangle_{L^2(\Omega_T)}\\
=&s\lambda \left[ \int_{\Omega} \varphi \pt z z \mathrm{d}x \right]_{-T}^T-s\lambda \int_{\Omega_T}\varphi|\pt z|^2\mathrm{d}x\mathrm{d}t+s\lambda d \int_{\Omega_T}\varphi |\nabla z|^2\mathrm{d}x\mathrm{d}t+\dfrac{s\lambda^2}{2}\int_{\Omega_T}\varphi|z|^2\left( \ptt\psi-d\Delta\psi \right)\mathrm{d}x\mathrm{d}t\\
&-\dfrac{s\lambda}{2}\left[ \int_{\Omega}\pt\varphi|z|^2\mathrm{d}x
\right]_{-T}^T+\dfrac{s\lambda^3}{2}\int_{\Omega_T}\varphi b(\psi) |z|^2\mathrm{d}x\mathrm{d}t+s^3\lambda^3\int_{\Omega_T}\varphi^3 b(\psi)|z|^2\mathrm{d}x\mathrm{d}t\\
&+s\lambda d\int_{\Gamma_T^1}\varphi\left( \dfrac{\lambda}{2}\pnu\psi |z|^2-z\pnu z \right)\mathrm{d}S\mathrm{d}t.
\end{align*}
Thus,
\begin{align}\label{Res2}
    \begin{split}
        &\left|s\lambda\int_{\Omega_T}\varphi|\pt z|^2 \mathrm{d}x\mathrm{d}t-s\lambda d\int_{\Gamma_{T}^1}\varphi\left( \dfrac{\lambda}{2}\pnu\psi|z|^2-z\pnu z \right)\mathrm{d}S\mathrm{d}t\right|\\
        \leq &\varepsilon\|P_1 z\|_{L^2(\Omega_T)}^2+C_{\varepsilon}s^2\lambda^2\int_{\Omega_T}\varphi^2|z|^2\mathrm{d}x\mathrm{d}t+s\lambda d\int_{\Omega_T}\varphi|\nabla z|^2\mathrm{d}x\mathrm{d}t+Cs^3\lambda^3\int_{\Omega_T}|b(\psi)|\varphi^3|z|^2\mathrm{d}x\mathrm{d}t\\
        &+\left|s\lambda \left[ \int_{\Omega} \varphi \pt z z \mathrm{d}x \right]_{-T}^T-\dfrac{s\lambda}{2}\left[ \int_{\Omega}\pt\varphi|z|^2\mathrm{d}x
 \right]_{-T}^T  \right|.
    \end{split}
\end{align}
Combining \eqref{Res1} and \eqref{Res2}, we obtain 
\begin{align}\label{Res3}
       & J_1+4\beta s\lambda d\int_{\Gamma_T^1}\varphi\left( \dfrac{\lambda}{2}\pnu\psi|z|^2-z\pnu z \right)\mathrm{d}S\mathrm{d}t\\
       &\geq 4d(\rho d-\beta)s\lambda\int_{\Omega_T}\varphi |\nabla z|^2\mathrm{d}x\mathrm{d}t-4\beta\left( Cs^3\lambda ^3 \int_{\Omega_T}\varphi^3|b(\psi)||z|^2\mathrm{d}x\mathrm{d}t+\varepsilon\|P_1 z\|_{L^2(\Omega)}^2+C_{\varepsilon}s^2\lambda^2\int_{\Omega_T}\varphi^2|z|^2\mathrm{d}x\mathrm{d}t \right) \notag\\
       &-4\beta\left|s\lambda \left[ \int_{\Omega} \varphi \pt z z \mathrm{d}x \right]_{-T}^T-\dfrac{s\lambda}{2}\left[ \int_{\Omega}\pt\varphi|z|^2\mathrm{d}x
 \right]_{-T}^T  \right|.  \notag
\end{align}
Using \eqref{Res2} again, we obtain
\begin{align}\label{Res4}
        &2(\rho d-\beta)s\lambda\int_{\Omega_T}\varphi |\pt z|^2\mathrm{d}x\mathrm{d}t-2(\rho d-\beta)\left|s\lambda \left[ \int_{\Omega} \varphi \pt z z \mathrm{d}x \right]_{-T}^T-\dfrac{s\lambda}{2}\left[ \int_{\Omega}\pt\varphi|z|^2\mathrm{d}x
 \right]_{-T}^T  \right|\notag\\
        &-2(\rho d-\beta)\left(  Cs^3\lambda^3 \int_{\Omega_T}\varphi^3|b(\psi)||z|^2\mathrm{d}x\mathrm{d}t+C_{\varepsilon}s^2\lambda^2\int_{\Omega_T}\varphi^2|z|^2\mathrm{d}x\mathrm{d}t+\varepsilon\|P_1 z\|_{L^2(\Omega_{T})}^2\right)\notag\\
        \leq &2(\rho d-\beta)s\lambda d\left[ \int_{\Omega_T}\varphi|\nabla z|^2\mathrm{d}x\mathrm{d}t +\int_{\Gamma_{T}^1}\varphi\left( \dfrac{\lambda}{2}\pnu\psi|z|^2-z\pnu z \right)\mathrm{d}S \mathrm{d}t\right].
\end{align}
Thus, by using \eqref{Res3} and \eqref{Res4}, we complete the proof of the claim \eqref{Res0}.

Next, we prove that
 \begin{align}\label{Res5}
     J_2\geq 2s^3\lambda^4\int_{\Omega_T}\varphi^3|b(\psi)|^2|z|^2\mathrm{d}x\mathrm{d}t+4s^3\lambda^3\int_{\Omega_T}\varphi^3|z|^2\left( d^2\rho|\nabla\psi|^2-\beta|\pt\psi|^2 \right)\mathrm{d}x\mathrm{d}t.
 \end{align}
We have
\begin{align*}
\begin{split}
        J_2=&2s^3\lambda^3\int_{\Omega_T}\varphi^3|z|^2\left( |\pt\psi|^2\ptt\psi+d^2\nabla^2\psi(\nabla\psi,\nabla\psi) \right)\mathrm{d}x\mathrm{d}t+2s^3\lambda^4\int_{\Omega_T}\varphi^3\left( |\pt\psi|^2-d|\nabla\psi|^2 \right)^2|z|^2\mathrm{d}x\mathrm{d}t.
    \end{split}
\end{align*}
Using the fact that $b(\psi)=|\pt\psi|^2-d|\nabla\psi|^2$ and $\nabla^2\psi(\nabla\psi,\nabla\psi)\geq 2\rho|\nabla\psi|^2$, we obtain \eqref{Res5}.

Now, by calculating the term $\left(\ptt-d\Delta \right)^2\varphi$, we obtain
\begin{align}\label{Res-1}
    |J_3|\leq C s \lambda^4 \int_{\Omega_T}\varphi |z|^2\mathrm{d}x\mathrm{d}t\leq C s \lambda^2 \int_{\Omega_T}\varphi^2 |z|^2\mathrm{d}x\mathrm{d}t .
\end{align}
In the following step, it remains to estimate $J_{1\Gamma}$, $J_{2\Gamma}$ and $J_{3\Gamma}$. We emphasize that $\nabla_\Gamma \varphi=0$ on $\Gamma^1$ and thus many terms involving the tangential derivatives of $\varphi$ vanish. Starting with $J_{1\Gamma}$, we have
\begin{align}
    \begin{split}\label{Res9}
        J_{1\Gamma}=&-4\beta s\lambda\int_{\Gamma_{T}^1}\varphi|\pt z|^2\mathrm{d}S\mathrm{d}t+2s\lambda^2\int_{\Gamma_{T}^1}\varphi|\pt\psi|^2|\pt z|^2\mathrm{d}S\mathrm{d}t \geq -4\beta s \lambda \int_{\Gamma_{T}^1}\varphi |\pt z|^2\mathrm{d}S\mathrm{d}t.
    \end{split}
\end{align}
Using a similar reasoning to $ \left\langle P_{1}z, s\lambda \varphi z \right\rangle_{L^2(\Omega_T)}$, we obtain
\begin{align*}
        &\left\langle P_{1\Gamma}z, s\lambda \varphi z \right\rangle_{L^2(\Gamma_{T}^1)}\\
        =&s\lambda \left[ \int_{\Gamma^1}\varphi \pt z z \mathrm{d}S \right]_{-T}^T-s\lambda \int_{\Gamma_{T}^1}\varphi |\pt z|^2\mathrm{d}S\mathrm{d}t+\delta s\lambda \int_{\Gamma_{T}^1}\varphi |\nabla_{\Gamma}z|^2\mathrm{d}S\mathrm{d}t+\dfrac{s\lambda^2}{2}\int_{\Gamma_{T}^1} \varphi \ptt\psi |z|^2 \mathrm{d}S\mathrm{d}t\\
        &\hspace{-0.3cm} -\dfrac{s\lambda}{2}\left[ \int_{\Gamma^1}\pt\varphi |z|^2 \mathrm{d}S \right]_{-T}^T+\dfrac{s\lambda^3}{2}\int_{\Gamma_{T}^1}\varphi|\pt \psi|^2 |z|^2  \mathrm{d}S\mathrm{d}t+s^3\lambda^3 \int_{\Gamma_{T}^1} \varphi^3 |\pt\psi|^2|z|^2 \mathrm{d}S\mathrm{d}t
        +ds\lambda \int_{\Gamma_{T}^1} \varphi z \pnu z \mathrm{d}S\mathrm{d}t.
        \end{align*}
        Thus,
\begin{align}\label{Resul2}
    \begin{split}
        &\left|s\lambda\int_{\Gamma_{T}^1}\varphi|\pt z|^2 \mathrm{d}S\mathrm{d}t-s\lambda d\int_{\Gamma_{T}^1}\varphi z\pnu z \mathrm{d}S\mathrm{d}t\right|\\
        \leq &\varepsilon\|P_{1\Gamma} z\|_{L^2(\Gamma_{T}^1)}^2+C_{\varepsilon}s^2\lambda^2\int_{\Gamma_{T}^1}\varphi^2|z|^2\mathrm{d}S\mathrm{d}t+s\lambda \delta\int_{\Gamma_{T}^1}\varphi|\nabla_{\Gamma} z|^2\mathrm{d}S\mathrm{d}t+\dfrac{3}{2}s^3\lambda^3\int_{\Gamma_{T}^1}|\pt\psi|^2\varphi^3|z|^2\mathrm{d}S\mathrm{d}t\\
        &+\left|s\lambda \left[ \int_{\Gamma^1} \varphi \pt z z \mathrm{d}S \right]_{-T}^T-\dfrac{s\lambda}{2}\left[ \int_{\Gamma^1}\pt\varphi|z|^2\mathrm{d}S
 \right]_{-T}^T  \right|.
    \end{split}
\end{align}
Combining \eqref{Res9} and \eqref{Resul2}, we obtain
\begin{align}\label{Resul3}
       & J_{1\Gamma}+4\beta s\lambda d\int_{\Gamma_T^1}\varphi z\pnu z \mathrm{d}S\mathrm{d}t \notag\\
       & \geq -4\beta s\lambda\delta\int_{\Gamma_{T}^1}\varphi |\nabla_{\Gamma} z|^2\mathrm{d}S\mathrm{d}t-4\beta\left( \dfrac{3}{2}s^3\lambda ^3 \int_{\Gamma_{T}^1}\varphi^3|\pt\psi|^2|z|^2\mathrm{d}S\mathrm{d}t+\varepsilon\|P_{1\Gamma} z\|_{L^2(\Gamma_{T}^1)}^2+C_{\varepsilon}s^2\lambda^2\int_{\Gamma_{T}^1}\varphi^2|z|^2\mathrm{d}S\mathrm{d}t \right) \notag\\
       &-4\beta\left|s\lambda \left[ \int_{\Gamma^1} \varphi \pt z z \mathrm{d}S \right]_{-T}^T-\dfrac{s\lambda}{2}\left[ \int_{\Gamma^1}\pt\varphi|z|^2\mathrm{d}S
 \right]_{-T}^T  \right|.
\end{align}
Using \eqref{Resul2} again, we obtain 
\begin{align}\label{Resul4}
&4\beta s\lambda d\int_{\Gamma_{T}^1}\varphi z\pnu z \mathrm{d}S\mathrm{d}t\notag\\
&\geq 4\beta s\lambda\int_{\Gamma_{T}^1}\varphi|\pt z|^2 \mathrm{d}S\mathrm{d}t-4\beta s\lambda \delta\int_{\Gamma_{T}^1}\varphi|\nabla_{\Gamma} z|^2\mathrm{d}S\mathrm{d}t-4\beta\left|s\lambda \left[ \int_{\Gamma^1} \varphi \pt z z \mathrm{d}S \right]_{-T}^T-\dfrac{s\lambda}{2}\left[ \int_{\Gamma^1}\pt\varphi|z|^2\mathrm{d}S
 \right]_{-T}^T  \right|\notag\\
&-4\beta \left( \varepsilon\|P_{1\Gamma} z\|_{L^2(\Gamma_{T}^1)}^2+C_{\varepsilon}s^2\lambda^2\int_{\Gamma_{T}^1}\varphi^2|z|^2\mathrm{d}S\mathrm{d}t +\dfrac{3}{2}s^3\lambda^3\int_{\Gamma_{T}^1}|\pt\psi|^2\varphi^3|z|^2\mathrm{d}S\mathrm{d}t\right).
\end{align}
By using \eqref{Resul3} and \eqref{Resul4}, we obtain
\begin{align}\label{Resul5}
&J_{1\Gamma}+8\beta s\lambda d\int_{\Gamma_{T}^1}\varphi z\pnu z \mathrm{d}S\mathrm{d}t\notag\\
&\geq 4\beta s\lambda\int_{\Gamma_{T}^1}\varphi|\pt z|^2 \mathrm{d}S\mathrm{d}t-8\beta s\lambda \delta\int_{\Gamma_{T}^1}\varphi|\nabla_{\Gamma} z|^2\mathrm{d}S\mathrm{d}t-8\beta\left|s\lambda \left[ \int_{\Gamma^1} \varphi \pt z z \mathrm{d}S \right]_{-T}^T-\dfrac{s\lambda}{2}\left[ \int_{\Gamma^1}\pt\varphi|z|^2\mathrm{d}S
 \right]_{-T}^T  \right|\notag\\
&-8\beta \left( \varepsilon\|P_{1\Gamma} z\|_{L^2(\Gamma_{T}^1)}^2+C_{\varepsilon}s^2\lambda^2\int_{\Gamma_{T}^1}\varphi^2|z|^2\mathrm{d}S\mathrm{d}t +\dfrac{3}{2}s^3\lambda^3\int_{\Gamma_{T}^1}|\pt\psi|^2\varphi^3|z|^2\mathrm{d}S\mathrm{d}t\right).
\end{align}
For the term $J_{2\Gamma}$, we have
\begin{align}\label{Res10}
    \begin{split}
    J_{2\Gamma}=&2s^3\lambda^4\int_{\Gamma_{T}^1}\varphi^3|\pt\psi|^4|z|^2\mathrm{d}S\mathrm{d}t-4\beta s^3\lambda^3\int_{\Gamma_{T}^1}\varphi^3|\pt\psi|^2|z|^2\mathrm{d}S\mathrm{d}t.
    \end{split}
\end{align}
Finally, we have 
$\displaystyle
    J_{3\Gamma}=-\dfrac{s}{2}\int_{\Gamma_{T}^1}|z|^2\left( \ptt-\delta\Delta_{\Gamma} \right)^2\varphi\,\mathrm{d}S\mathrm{d}t.
$
Similarly to the calculation for $J_{3}$, we find
\begin{align}\label{Res11}
    J_{3\Gamma}\leq C s\lambda^4\int_{\Gamma_{T}^1}\varphi |z|^2\mathrm{d}S\mathrm{d}t\leq C s\lambda^2\int_{\Gamma_{T}^1}\varphi^2 |z|^2\mathrm{d}S\mathrm{d}t.
\end{align}
By collecting the estimates \eqref{Res0}, \eqref{Res5}, \eqref{Res-1}, and \eqref{Res9}-\eqref{Res11}, we obtain
\begin{align}\label{Res-2}
&\mathcal{J\geq}2(\rho d-\beta )s\lambda\int_{\Omega_T}\varphi \left( |\pt z|^2+d|\nabla z|^2 \right)\mathrm{d}x\mathrm{d}t+2s^3\lambda^4\int_{\Omega_T}\varphi^3 |b(\psi)|^2|z|^2\mathrm{d}x\mathrm{d}t\notag\\
        &+4s^3\lambda^3\int_{\Omega_T}\varphi^3|z|^2\left( d^2\rho|\nabla \psi|^2-\beta|\pt\psi|^2 \right)\mathrm{d}x\mathrm{d}t-C\left( s^3\lambda^3\int_{\Omega_T}\varphi^3 |z|^2|b(\psi)|\mathrm{d}x\mathrm{d}t\right.\notag\\
        &\left.+\varepsilon\|P_1 z\|_{L^2(\Omega_{T})}^2+C_\varepsilon s^2\lambda^2\int_{\Omega_T}\varphi^2|z|^2\mathrm{d}x\mathrm{d}t \right)-C\left|s\lambda \left[ \int_{\Omega} \varphi \pt z z \mathrm{d}x \right]_{-T}^T-\dfrac{s\lambda}{2}\left[ \int_{\Omega}\pt\varphi|z|^2\mathrm{d}x
 \right]_{-T}^T  \right|\notag\\
 &+ 4\beta s\lambda\int_{\Gamma_{T}^1}\varphi|\pt z|^2 \mathrm{d}S\mathrm{d}t-8\beta s\lambda \delta\int_{\Gamma_{T}^1}\varphi|\nabla_{\Gamma} z|^2\mathrm{d}S\mathrm{d}t+2s^3\lambda^4\int_{\Gamma_{T}^1}\varphi^3|\pt\psi|^4|z|^2\mathrm{d}S\mathrm{d}t\\
 &-4\beta s^3\lambda^3\int_{\Gamma_{T}^1}\varphi^3|\pt\psi|^2|z|^2\mathrm{d}S\mathrm{d}t-8\beta\left|s\lambda \left[ \int_{\Gamma^1} \varphi \pt z z \mathrm{d}S \right]_{-T}^T-\dfrac{s\lambda}{2}\left[ \int_{\Gamma^1}\pt\varphi|z|^2\mathrm{d}S
 \right]_{-T}^T  \right|\notag\\
&-8\beta \left( \varepsilon\|P_{1\Gamma} z\|_{L^2(\Gamma_{T}^1)}^2+C_{\varepsilon}s^2\lambda^2\int_{\Gamma_{T}^1}\varphi^2|z|^2\mathrm{d}S\mathrm{d}t +\dfrac{3}{2}s^3\lambda^3\int_{\Gamma_{T}^1}|\pt\psi|^2\varphi^3|z|^2\mathrm{d}S\mathrm{d}t\right)\notag.
\end{align}

Since $\beta d<\rho d^2$, there exists a small positive constant $\eta$ such that
$\beta d+\beta\eta<\rho d^2.$
Next, we denote 
$ \Omega_{T}^{\eta}=\left\lbrace (x,t)\in\Omega_T \colon\;\; |b(\psi)|\leq \eta |\nabla\psi|^2 \right\rbrace,$
where $b(\psi)$ is defined by \eqref{bpsi}. By \eqref{Res-2}, we obtain
\begin{align*}
    \begin{split}
        &\mathcal{J}\geq 2(\rho d-\beta )s\lambda\int_{\Omega_T}\varphi \left( |\pt z|^2+d|\nabla z|^2 \right)\mathrm{d}x\mathrm{d}t+2s^3\lambda^4\int_{\Omega_{T}\setminus\Omega_{T}^\eta}\varphi^3 |b(\psi)|^2|z|^2\mathrm{d}x\mathrm{d}t\\
        &+4s^3\lambda^3\int_{\Omega_{T}^\eta}\varphi^3|z|^2\left( d^2\rho|\nabla \psi|^2-\beta|\pt\psi|^2 \right)\mathrm{d}x\mathrm{d}t+4s^3\lambda^3\int_{\Omega_{T}\setminus\Omega_{T}^\eta}\varphi^3|z|^2\left( d^2\rho|\nabla \psi|^2-\beta|\pt\psi|^2 \right)\mathrm{d}x\mathrm{d}t\\
        &-C\left( s^3\lambda^3\int_{\Omega_{T}^\eta}\varphi^3 |z|^2|b(\psi)|\mathrm{d}x\mathrm{d}t+s^3\lambda^3\int_{\Omega_{T}\setminus\Omega_{T}^\eta}\varphi^3 |z|^2|b(\psi)|\mathrm{d}x\mathrm{d}t+\varepsilon\|P_1 z\|_{L^2(\Omega_{T})}^2+C_\varepsilon s^2\lambda^2\int_{\Omega_T}\varphi^2|z|^2\mathrm{d}x\mathrm{d}t \right)\\
        &-C\left|s\lambda \left[ \int_{\Omega} \varphi \pt z z \mathrm{d}x \right]_{-T}^T-\dfrac{s\lambda}{2}\left[ \int_{\Omega}\pt\varphi|z|^2\mathrm{d}x
 \right]_{-T}^T  \right|+ 4\beta s\lambda\int_{\Gamma_{T}^1}\varphi|\pt z|^2 \mathrm{d}S\mathrm{d}t-8\beta s\lambda \delta\int_{\Gamma_{T}^1}\varphi|\nabla_{\Gamma} z|^2\mathrm{d}S\mathrm{d}t\\
  &+2s^3\lambda^4\int_{\Gamma_{T}^1}\varphi^3|\pt\psi|^4|z|^2\mathrm{d}S\mathrm{d}t-4\beta s^3\lambda^3\int_{\Gamma_{T}^1}\varphi^3|\pt\psi|^2|z|^2\mathrm{d}S\mathrm{d}t\\
  &-8\beta\left|s\lambda \left[ \int_{\Gamma^1} \varphi \pt z z \mathrm{d}S \right]_{-T}^T-\dfrac{s\lambda}{2}\left[ \int_{\Gamma^1}\pt\varphi|z|^2\mathrm{d}S
 \right]_{-T}^T  \right|\\
&-8\beta \left( \varepsilon\|P_{1\Gamma} z\|_{L^2(\Gamma_{T}^1)}^2+C_{\varepsilon}s^2\lambda^2\int_{\Gamma_{T}^1}\varphi^2|z|^2\mathrm{d}S\mathrm{d}t +\dfrac{3}{2}s^3\lambda^3\int_{\Gamma_{T}^1}|\pt\psi|^2\varphi^3|z|^2\mathrm{d}S\mathrm{d}t\right).
    \end{split}
\end{align*}
By $b(\psi)=|\pt\psi|^2-d|\nabla\psi|^2$ and $b(\psi)\leq \eta |\nabla\psi|^2$ in $\Omega_{T}^{\eta}$, we obtain $d^2\rho|\nabla\psi|^2-\beta|\pt\psi|^2\geq [d^2\rho-\beta(\eta+d)]|\nabla\psi|^2.$
Hence, 
\begin{align*}
&\mathcal{J}\geq 2(\rho d-\beta )s\lambda\int_{\Omega_T}\varphi \left( |\pt z|^2+d|\nabla z|^2 \right)\mathrm{d}x\mathrm{d}t+2s^3\lambda^4\eta^2\int_{\Omega_{T}\setminus\Omega_{T}^\eta}\varphi^3 |\nabla \psi|^2|z|^2\mathrm{d}x\mathrm{d}t\\
&+4[d^2\rho-\beta(\eta+d)]s^3\lambda^3\int_{\Omega_{T}^\eta}\varphi^3|z|^2|\nabla\psi|^2\mathrm{d}x\mathrm{d}t-C\left( \eta s^3\lambda^3\int_{\Omega_{T}^\eta}\varphi^3 |z|^2\mathrm{d}x\mathrm{d}t\right.\\&\left.+s^3\lambda^3\int_{\Omega_{T}\setminus\Omega_{T}^\eta}\varphi^3 |z|^2\mathrm{d}x\mathrm{d}t+\varepsilon\|P_1 z\|_{L^2(\Omega_{T})}^2+C_\varepsilon s^2\lambda^2\int_{\Omega_T}\varphi^2|z|^2\mathrm{d}x\mathrm{d}t \right)\\
&-C\left|s\lambda \left[ \int_{\Omega} \varphi \pt z z \mathrm{d}x \right]_{-T}^T-\dfrac{s\lambda}{2}\left[ \int_{\Omega}\pt\varphi|z|^2\mathrm{d}x\right]_{-T}^T  \right|\\
&+2s^3\lambda^4\int_{\Gamma_{T}^1}\varphi^3|\pt\psi|^4|z|^2\mathrm{d}S\mathrm{d}t-4\beta s^3\lambda^3\int_{\Gamma_{T}^1}\varphi^3|\pt\psi|^2|z|^2\mathrm{d}S\mathrm{d}t+ 4\beta s\lambda\int_{\Gamma_{T}^1}\varphi|\pt z|^2 \mathrm{d}S\mathrm{d}t\\
&-8\beta s\lambda \delta\int_{\Gamma_{T}^1}\varphi|\nabla_{\Gamma} z|^2\mathrm{d}S\mathrm{d}t+2s^3\lambda^4\int_{\Gamma_{T}^1}\varphi^3|\pt\psi|^4|z|^2\mathrm{d}S\mathrm{d}t-4\beta s^3\lambda^3\int_{\Gamma_{T}^1}\varphi^3|\pt\psi|^2|z|^2\mathrm{d}S\mathrm{d}t\\
  &-8\beta\left|s\lambda \left[ \int_{\Gamma^1} \varphi \pt z z \mathrm{d}S \right]_{-T}^T-\dfrac{s\lambda}{2}\left[ \int_{\Gamma^1}\pt\varphi|z|^2\mathrm{d}S
 \right]_{-T}^T  \right|\\
&-8\beta \left( \varepsilon\|P_{1\Gamma} z\|_{L^2(\Gamma_{T}^1)}^2+C_{\varepsilon}s^2\lambda^2\int_{\Gamma_{T}^1}\varphi^2|z|^2\mathrm{d}S\mathrm{d}t +\dfrac{3}{2}s^3\lambda^3\int_{\Gamma_{T}^1}|\pt\psi|^2\varphi^3|z|^2\mathrm{d}S\mathrm{d}t\right).
\end{align*}
Since $|\nabla\psi|>0$ in $\overline{\Omega}$, then there exists a constant $C_{2}>0$ such that $|\nabla \psi|^2 \geq C_{2}$. By using the fact that $\beta d+\beta \eta<\rho d^2$, we obtain
\begin{align*}
        &\mathcal{J}\geq \dfrac{2\beta\eta}{d}s\lambda\int_{\Omega_T}\varphi \left( |\pt z|^2+d|\nabla z|^2 \right)\mathrm{d}x\mathrm{d}t+(2\lambda\eta^2C_2-C)s^3\lambda^3\int_{\Omega_{T}\setminus\Omega_{T}^\eta}\varphi^3 |z|^2\mathrm{d}x\mathrm{d}t\\
        &\hspace{-0.7cm} +\left[4C_{2}(d^2\rho-\beta(\eta+d))-\eta C\right] s^3\lambda^3\int_{\Omega_{T}^\eta}\varphi^3|z|^2\mathrm{d}x\mathrm{d}t-C\left( \varepsilon\|P_1 z\|_{L^2(\Omega_{T})}^2+C_\varepsilon s^2\lambda^2\int_{\Omega_T}\varphi^2|z|^2\mathrm{d}x\mathrm{d}t \right)\\
        &-C\left|s\lambda \left[ \int_{\Omega} \varphi \pt z z \mathrm{d}x \right]_{-T}^T-\dfrac{s\lambda}{2}\left[ \int_{\Omega}\pt\varphi|z|^2\mathrm{d}x
 \right]_{-T}^T  \right|+4\beta s\lambda\int_{\Gamma_{T}^1}\varphi |\pt z|^2\mathrm{d}S\mathrm{d}t\\
&-8\beta s\lambda \delta\int_{\Gamma_{T}^1}\varphi|\nabla_{\Gamma} z|^2\mathrm{d}S\mathrm{d}t+s^3\lambda^3\int_{\Gamma_{T}^1}\varphi^3|\pt\psi|^2(2\lambda |\pt \psi|^2-16\beta)|z|^2\mathrm{d}S\mathrm{d}t\\
&-8\beta\left|s\lambda \left[ \int_{\Gamma^1} \varphi \pt z z \mathrm{d}S \right]_{-T}^T-\dfrac{s\lambda}{2}\left[ \int_{\Gamma^1}\pt\varphi|z|^2\mathrm{d}S
 \right]_{-T}^T  \right| 
-8\beta \left( \varepsilon\|P_{1\Gamma} z\|_{L^2(\Gamma_{T}^1)}^2+C_{\varepsilon}s^2\lambda^2\int_{\Gamma_{T}^1}\varphi^2|z|^2\mathrm{d}S\mathrm{d}t \right).
\end{align*}
Choosing $\varepsilon \le\min\left(\dfrac{1}{2C},\dfrac{1}{16\beta}\right)$, taking small $\eta$ and large $\lambda\geq\lambda_1$, we deduce
\begin{align}\label{Res-3}
        &\mathcal{J}\geq C s\lambda \int_{\Omega_T} \varphi\left( |\pt z|^2+|\nabla z|^2+\varphi^2s^2\lambda^2|z|^2 \right)\mathrm{d}x\mathrm{d}t+C s \lambda\int_{\Gamma_{T}^1}\varphi |\pt z|^2\mathrm{d}S\mathrm{d}t \notag\\
        &-8\beta s\lambda \delta\int_{\Gamma_{T}^1}\varphi|\nabla_{\Gamma} z|^2\mathrm{d}S\mathrm{d}t-\dfrac{1}{2}\|P_1 z\|_{L^2(\Omega_{T})}^2-\dfrac{1}{2}\|P_{1\Gamma} z\|_{L^2(\Gamma_{T}^1)}^2-Cs^2\lambda^2\int_{\Gamma_{T}^1}\varphi^2|z|^2\mathrm{d}S\mathrm{d}t \notag\\ &+s^3\lambda^3\int_{\Gamma_{T}^1}\varphi^3|\pt\psi|^2(2\lambda |\pt \psi|^2-16\beta)|z|^2\mathrm{d}S\mathrm{d}t-C\left|s\lambda \left[ \int_{\Omega} \varphi \pt z z \mathrm{d}x \right]_{-T}^T-\dfrac{s\lambda}{2}\left[ \int_{\Omega}\pt\varphi|z|^2\mathrm{d}x
 \right]_{-T}^T  \right| \notag\\
 &-C\left|s\lambda \left[ \int_{\Gamma^1} \varphi \pt z z \mathrm{d}S \right]_{-T}^T-\dfrac{s\lambda}{2}\left[ \int_{\Gamma^1}\pt\varphi|z|^2\mathrm{d}S
 \right]_{-T}^T  \right|.
\end{align}
Combining \eqref{Res-3} and Young's inequality, we finally obtain \eqref{Res-3.1}.

It remains to estimate $B_0$ defined in \eqref{B0}. Owing to the fact that $\psi(x,t)=\psi_{0}(x)-\beta t^2+C_1$, we obtain $\pt\psi=-2\beta t$. Therefore, 
\begin{align}\label{Res14}
        &-s\lambda\left[ \int_{\Omega}\pt\psi\varphi|\pt z|^2 \mathrm{d}x\right]_{-T}^T-s\lambda\left[ \int_{\Gamma^1}\pt\psi \varphi|\pt z|^2\mathrm{d}S \right]_{-T}^T-ds\lambda\left[ \int_{\Omega}\varphi \pt\psi |\nabla z|^2 \mathrm{d}x\right]_{-T}^T \notag\\
        &-\delta s \lambda \left[ \int_{\Gamma^1}\varphi |\nabla_{\Gamma}z|^2\pt\psi\mathrm{d}S \right]_{-T}^T \notag\\
        =&2s\lambda\beta T\int_{\Omega} \varphi(x,T)|\pt z(x,T)|^2\mathrm{d}x+2s\lambda\beta T\int_{\Omega} \varphi(x,-T)|\pt z(x,-T)|^2\mathrm{d}x\notag\\
        &+2s\lambda\beta T\int_{\Gamma^1} \varphi(x,T)|\pt z(x,T)|^2\mathrm{d}S+2s\lambda\beta T\int_{\Omega} \varphi(x,-T)|\pt z(x,-T)|^2\mathrm{d}S \notag\\
        &+2ds\lambda\beta T\int_{\Omega}\varphi(x,T)|\nabla z(x,T)|^2\mathrm{d}x+2ds\lambda\beta T\int_{\Omega}\varphi(x,-T)|\nabla z(x,-T)|^2\mathrm{d}x \notag\\
        &+2\delta s\lambda\beta T\int_{\Gamma^1}\varphi(x,T)|\nabla_{\Gamma} z(x,T)|^2\mathrm{d}S+2\delta s\lambda\beta T\int_{\Gamma^1}\varphi(x,-T)|\nabla_{\Gamma} z(x,-T)|^2\mathrm{d}S \geq 0.
\end{align}
Using the Young inequality and the fact that $\psi\in C^4(\overline{\Omega_T})$, we obtain
\begin{align}\label{res14}
          &\left|+2sd\lambda \int_{\Omega}\varphi\pt z(x,\tau) \nabla\psi(x,\tau)\cdot\nabla z(x,\tau)\mathrm{d}x\right.
          -s\lambda \int_{\Omega}\varphi \pt z(x,\tau) z(x,\tau) (\pt^2\psi(x,\tau)-d\Delta\psi(x,\tau))\mathrm{d}x\notag\\
          &-s\lambda^2\int_{\Omega}\varphi \pt z(x,\tau) z(x,\tau)(|\pt \psi(x,\tau)|^2-d|\nabla\psi(x,\tau)|^2)\mathrm{d}x \notag\\
       &-s^3\lambda^3\int_{\Omega}\varphi(x,\tau)^3|z(x,\tau)|^2\pt\psi(x,\tau) (|\pt\psi(x,\tau)|^2-d|\nabla\psi(x,\tau)|^2)\mathrm{d}x\notag\\
       &+\dfrac{s}{2}\lambda^2\int_{\Omega}\varphi(x,\tau) |z(x,\tau)|^2\pt\psi(x,\tau) (\pt^2\psi(x,\tau)+\lambda |\pt\psi(x,\tau)|^2) \mathrm{d}x\notag\\
       &+\dfrac{s}{2}\lambda\int_{\Omega}\varphi(x,\tau) |z(x,\tau)|^2(\pt^3\psi(x,\tau)+2\lambda\pt\psi(x,\tau)\pt^2\psi(x,\tau)) \mathrm{d}x\notag\\
       &-s\lambda \int_{\Gamma^1}\varphi(x,\tau) \pt z(x,\tau) z(x,\tau) (\pt^2\psi(x,\tau)+d\pnu\psi(x,\tau))\mathrm{d}S\notag\\
       &-s\lambda^2 \int_{\Gamma^1} \varphi (x,\tau)\pt z(x,\tau) z(x,\tau)|\pt\psi(x,\tau)|^2 \mathrm{d}S-s^3\lambda^3 \int_{\Gamma^1}\varphi(x,\tau)^3 |z|^2 (\pt\psi(x,\tau))^3 \mathrm{d}S\notag\\
       &\left.+\dfrac{s}{2}\lambda^2\int_{\Gamma^1}\varphi(x,\tau) |z|^2 \pt\psi(x,\tau)(3\pt^2\psi(x,\tau)+\lambda|\pt\psi(x,\tau)|^2)\mathrm{d}S \right|\notag\\
       \leq & Cs\lambda\int_{\Omega} \varphi(x,\tau)|\pt z(x,\tau)|^2\mathrm{d}x+Cs\lambda\int_{\Omega} \varphi(x,\tau)|\nabla z(x,\tau)|^2\mathrm{d}x+Cs^3\lambda^3\int_{\Omega} \varphi^3(x,\tau)|z(x,\tau)|^2\mathrm{d}x\notag\\
       &+Cs\lambda \int_{\Gamma^1}\varphi(x,\tau)|\pt z(x,\tau)|^2\mathrm{d}S+Cs^3\lambda^3\int_{\Gamma^1}\varphi^3(x,\tau)|z(x,\tau)|^2\mathrm{d}S
\end{align}
for $\tau=-T,T$ and $\lambda$ large enough. On the other hand, by \eqref{Inclu}, we obtain
\begin{align}\label{result14}
   & -d^2s\lambda\int_{\Gamma_{T}}\varphi\pnu\psi|\pnu z|^2\mathrm{d}S\mathrm{d}t\notag\\
    &= -d^2s\lambda\int_{\Gamma_{T}^1}\varphi\pnu\psi|\pnu z|^2\mathrm{d}S\mathrm{d}t-d^2s\lambda\int_{\Gamma^0\backslash\gamma\times (-T,T)}\varphi\pnu\psi|\pnu z|^2\mathrm{d}S\mathrm{d}t -d^2s\lambda\int_{\gamma_{T}}\varphi\pnu\psi|\pnu z|^2\mathrm{d}S\mathrm{d}t\notag\\
    &\geq-d^2s\lambda\int_{\Gamma_{T}^1}\varphi\pnu\psi|\pnu z|^2\mathrm{d}S\mathrm{d}t-d^2s\lambda\int_{\gamma_{T}}\varphi\pnu\psi|\pnu z|^2\mathrm{d}S\mathrm{d}t.
\end{align}
Gathering the estimates \eqref{Res14}-\eqref{result14} and using the fact that $\psi\in C^4(\overline{\Omega_T})$, $\pnu\psi<0$ on $\Gamma^1$, and $|\nabla \psi|>0$ in $\overline{\Omega}$, we obtain
\begin{align}\label{Res15}
        B_0\geq&\, C s\lambda\int_{\Gamma_{T}^1}\varphi|\pnu z|^2\mathrm{d}S\mathrm{d}t-Cs\lambda\int_{\gamma_{T}}\varphi|\pnu z|^2\mathrm{d}S\mathrm{d}t+d(d-\delta)s\lambda \int_{\Gamma_{T}^1}\varphi\pnu\psi|\nabla_\Gamma z|^2\mathrm{d}S\mathrm{d}t\notag\\
        & \hspace{-1cm} +C_{*}s^3\lambda^3\int_{\Gamma_{T}^1}\varphi^3|z|^2\mathrm{d}S\mathrm{d}t-Cs\lambda^2\int_{\Gamma_{T}^1}\varphi |z|^2\mathrm{d}S\mathrm{d}t-Cs\lambda\int_{\Gamma_{T}^1}\varphi |z|^2\mathrm{d}S\mathrm{d}t\notag\\
        &\hspace{-1cm} -ds\lambda \int_{\Gamma_{T}^1}z\pnu z \varphi \left( d\pnu\psi+d\Delta\psi+d\lambda|\nabla\psi|^2 \right) \mathrm{d}S\mathrm{d}t- Cs\lambda\int_{\Omega} \varphi(x,T)|\pt z(x,T)|^2\mathrm{d}x\notag\\
        &\hspace{-1cm} - Cs\lambda\int_{\Omega} \varphi(x,-T)|\pt z(x,-T)|^2\mathrm{d}x-Cs\lambda\int_{\Omega} \varphi(x,T)|\nabla z(x,T)|^2\mathrm{d}x-Cs\lambda\int_{\Omega} \varphi(x,-T)|\nabla z(x,-T)|^2\mathrm{d}x\notag\\
        &\hspace{-1cm} -Cs^3\lambda^3\int_{\Omega} \varphi^3(x,T)|z(x,T)|^2\mathrm{d}x-Cs^3\lambda^3\int_{\Omega} \varphi^3(x,-T)|z(x,-T)|^2\mathrm{d}x-Cs\lambda \int_{\Gamma^1}\varphi(x,T)|\pt z(x,T)|^2\mathrm{d}S\notag\\
        &\hspace{-1cm}  -Cs\lambda \int_{\Gamma^1}\varphi(x,-T)|\pt z(x,-T)|^2\mathrm{d}S-Cs^3\lambda^3\int_{\Gamma^1}\varphi^3(x,T)|z(x,T)|^2\mathrm{d}S\notag\\
        &\hspace{-0cm} -Cs^3\lambda^3\int_{\Gamma^1}\varphi^3(x,-T)|z(x,-T)|^2\mathrm{d}S.
\end{align}
Furthermore, using the Young inequality and the fact that $\psi\in C^4(\overline{\Omega_T})$, we obtain
    \begin{align}\label{Res16}
            \left|d^2 s\lambda \int_{\Gamma_{T}^1}z\pnu z \varphi \left( \pnu\psi+\Delta\psi+\lambda|\nabla\psi|^2 \right)\mathrm{d}S\mathrm{d}t  \right|
            \leq Cs^2 \lambda^3 \int_{\Gamma_{T}^1}\varphi |z|^2\mathrm{d}S\mathrm{d}t+C\lambda \int_{\Gamma_{T}^1}\varphi |\pnu z|^2\mathrm{d}S\mathrm{d}t
    \end{align}
for all $\lambda \geq \lambda_1$. Therefore, using \eqref{Res15} and \eqref{Res16}, we obtain
    \begin{align}\label{Res17}
            B_0\geq& C s\lambda\int_{\Gamma_{T}^1}\varphi|\pnu z|^2\mathrm{d}S\mathrm{d}t-C\lambda\int_{\Gamma_{T}^1}\varphi |\pnu z|^2\mathrm{d}S\mathrm{d}t -C s\lambda \int_{\gamma_{T}}\varphi|\pnu z|^2\mathrm{d}S\mathrm{d}t \notag\\
            &\hspace{-1cm} +C_{*} s^3\lambda^3\int_{\Gamma_{T}^1}\varphi^3 |z|^2\mathrm{d}S\mathrm{d}t-Cs^2\lambda^3 \int_{\Gamma_{T}^1}\varphi |z|^2\mathrm{d}S\mathrm{d}t+d(\delta-d)s\lambda\int_{\Gamma_{T}^1}|\pnu \psi|\varphi|\nabla_{\Gamma}z|^2\mathrm{d}S \mathrm{d}t \notag\\
            &\hspace{-1cm} - Cs\lambda\int_{\Omega} \varphi(x,T)|\pt z(x,T)|^2\mathrm{d}x- Cs\lambda\int_{\Omega} \varphi(x,-T)|\pt z(x,-T)|^2\mathrm{d}x-Cs\lambda\int_{\Omega} \varphi(x,T)|\nabla z(x,T)|^2\mathrm{d}x \notag\\
        &\hspace{-1cm} -Cs\lambda\int_{\Omega} \varphi(x,-T)|\nabla z(x,-T)|^2\mathrm{d}x-Cs^3\lambda^3\int_{\Omega} \varphi^3(x,T)|z(x,T)|^2\mathrm{d}x-Cs^3\lambda^3\int_{\Omega} \varphi^3(x,-T)|z(x,-T)|^2\mathrm{d}x \notag\\
        &\hspace{-1cm}  -Cs\lambda \int_{\Gamma^1}\varphi(x,T)|\pt z(x,T)|^2\mathrm{d}S-Cs\lambda \int_{\Gamma^1}\varphi(x,-T)|\pt z(x,-T)|^2\mathrm{d}S-Cs^3\lambda^3\int_{\Gamma^1}\varphi^3(x,T)|z(x,T)|^2\mathrm{d}S\notag\\
        &-Cs^3\lambda^3\int_{\Gamma^1}\varphi^3(x,-T)|z(x,-T)|^2\mathrm{d}S 
    \end{align}
for large $\lambda\geq \lambda_1$ and large $s\geq s_1$. On the other hand, we have
    \[ C s^2\lambda^3\int_{\Gamma_{T}^1} \varphi|z|^2\mathrm{d}S\mathrm{d}t\leq \dfrac{C_{*}s^3\lambda^3}{2}\int_{\Gamma_{T}^1}\varphi^3 |z|^2\mathrm{d}S\mathrm{d}t \quad \text{and} \quad 
    C\lambda\int_{\Gamma_{T}^1}\varphi|\pnu z|^2\mathrm{d}S\mathrm{d}t \leq \dfrac{Cs\lambda}{2}\int_{\Gamma_{T}^1}\varphi |\pnu z|^2\mathrm{d}S\mathrm{d}t. \]
    Hence,
    \begin{align}\label{Res18}
s^3\lambda^3\int_{\Gamma_{T}^1}\varphi^3|z|^2\mathrm{d}S\mathrm{d}t-s^2\lambda^3\int_{\Gamma_{T}^1}\varphi^3|z|^2\mathrm{d}S\mathrm{d}t\geq \dfrac{s^3\lambda^3}{2}\int_{\Gamma_{T}^1}\varphi^3|z|^2\mathrm{d}S\mathrm{d}t,
    \end{align}
    and
    \begin{align}\label{Res19}
        s\lambda \int_{\Gamma_{T}^1}\varphi |\pnu z|^2\mathrm{d}S\mathrm{d}t-\lambda \int_{\Gamma_{T}^1}\varphi |\pnu z|^2\mathrm{d}S\mathrm{d}t \geq \dfrac{s\lambda}{2}\int_{\Gamma_{T}^1}\varphi |\pnu z|^2\mathrm{d}S\mathrm{d}t.
    \end{align}
    By adding up \eqref{Res17}-\eqref{Res19},  we conclude that
    \begin{align}\label{Res20}
            B_0\geq& Cs\lambda \int_{\Gamma_{T}^1}\varphi |\pnu z|^2\mathrm{d}S\mathrm{d}t-Cs\lambda\int_{\gamma_{T}}\varphi|\pnu z|^2\mathrm{d}S\mathrm{d}t+C_3s^3\lambda^3\int_{\Gamma_{T}^1}\varphi^3|z|^2\mathrm{d}S\mathrm{d}t  \notag\\
            &+C^\prime d(\delta-d)s\lambda\int_{\Gamma_{T}^1}\varphi|\nabla_{\Gamma}z|^2\mathrm{d}S\mathrm{d}t-Cs\lambda\int_{\Omega}\varphi(x,T)(|\pt z(x,T)|^2+|\nabla z(x,T)|^2)\mathrm{d}x  \notag\\
            &-Cs\lambda\int_{\Omega}\varphi(x,-T)(|\pt z(x,-T)|^2+|\nabla z(x,-T)|^2)\mathrm{d}x-Cs^3\lambda^3\int_{\Omega}\varphi^3(x,T)|z(x,T)|^2\mathrm{d}x  \notag\\
            &-C s^3\lambda^3\int_{\Omega}\varphi^3(x,-T)|z(x,-T)|^2\mathrm{d}x-Cs\lambda \int_{\Gamma^1}\varphi(x,T)|\pt z(x,T)|^2\mathrm{d}S  \notag\\
            &-Cs\lambda \int_{\Gamma^1}\varphi(x,-T)|\pt z(x,-T)|^2\mathrm{d}S-Cs^3\lambda^3\int_{\Gamma^1}\varphi^3(x,T)|z(x,T)|^2\mathrm{d}S  \notag\\
        & -Cs^3\lambda^3\int_{\Gamma^1}\varphi^3(x,-T)|z(x,-T)|^2\mathrm{d}S,
    \end{align}
    where $C_3=\dfrac{C_*}{2}$ and $C^\prime$ is defined in \eqref{Inclu}. Combining \eqref{Res-3.1} and \eqref{Res20}, we obtain
    \begin{align}\label{Res21}
           &\left\langle P_{1}z,P_{2}z\right\rangle_{L^2(\Omega_T)}+\left\langle P_{1\Gamma}z,P_{2\Gamma}z \right\rangle_{L^2(\Gamma_{T}^1)}  \notag\\
           \geq& Cs\lambda\int_{\Omega_T}\varphi\left( |\pt z|^2+|\nabla z|^2 \right)\mathrm{d}x\mathrm{d}t+C s^3\lambda^3\int_{\Omega_T}\varphi^3 |z|^2\mathrm{d}x\mathrm{d}t+[C^\prime d(\delta-d)-8\beta \delta]s\lambda \int_{\Gamma_{T}^1}\varphi|\nabla_{\Gamma}z|^2\mathrm{d}S\mathrm{d}t \notag\\
           &\hspace{-0.5cm} +C_{3}s^3\lambda^3\int_{\Gamma_{T}^1}\varphi^3|z|^2\mathrm{d}S\mathrm{d}t+Cs\lambda\int_{\Gamma_{T}^1}\varphi |\pnu z|^2\mathrm{d}S\mathrm{d}t-Cs\lambda\int_{\gamma_{T}}\varphi |\pnu z|^2\mathrm{d}S\mathrm{d}t+C s\lambda\int_{\Gamma_{T}^1}\varphi |\pt z|^2\mathrm{d}S\mathrm{d}t \notag\\
           &\hspace{-0.5cm} -2(d\rho+\beta)\dfrac{\lambda}{2}\int_{\Gamma_{T}^1}\pnu\psi |z|^2\mathrm{d}S\mathrm{d}t-2(3\beta-d\rho)\int_{\Gamma_{T}^1}\varphi z \pnu z\mathrm{d}S\mathrm{d}t\notag\\
           &\hspace{-0.5cm} -\dfrac{1}{2}\|P_{1}z\|_{L^2(\Omega_{T})}^2 -\dfrac{1}{2}\|P_{1\Gamma} z\|_{L^2(\Gamma_{T}^1)}^2+s^3\lambda^3\int_{\Gamma_{T}^1}\varphi^3|\pt\psi|^2(2\lambda |\pt \psi|^2-16\beta)|z|^2\mathrm{d}S\mathrm{d}t\notag\\
           &\hspace{-0.5cm}-Cs\lambda\int_{\Omega}\varphi(x,T)(|\pt z(x,T)|^2+|\nabla z(x,T)|^2)\mathrm{d}x-Cs^3\lambda^3\int_{\Omega}\varphi^3(x,T)|z(x,T)|^2\mathrm{d}x \notag\\
           &\hspace{-0.5cm} -Cs\lambda\int_{\Omega}\varphi(x,-T)(|\pt z(x,-T)|^2+|\nabla z(x,-T)|^2)\mathrm{d}x-C s^3\lambda^3\int_{\Omega}\varphi^3(x,-T)|z(x,-T)|^2\mathrm{d}x \notag\\
           &\hspace{-0.5cm} -Cs\lambda \int_{\Gamma^1}\varphi(x,T)|\pt z(x,T)|^2\mathrm{d}S-Cs\lambda \int_{\Gamma^1}\varphi(x,-T)|\pt z(x,-T)|^2\mathrm{d}S \notag\\
           &\hspace{-0cm}  -Cs^3\lambda^3\int_{\Gamma^1}\varphi^3(x,T)|z(x,T)|^2\mathrm{d}S-Cs^3\lambda^3\int_{\Gamma^1}\varphi^3(x,-T)|z(x,-T)|^2\mathrm{d}S
    \end{align}
for $\lambda\geq \lambda_1$ and $s\geq s_1$. Since $\pnu\psi<0$ on $\Gamma^1$, then $\displaystyle -2(d\rho+\beta)\dfrac{\lambda}{2}\int_{\Gamma_{T}^1}\pnu\psi |z|^2\mathrm{d}S\mathrm{d}t\geq 0$.\\
By using Young's inequality, we obtain
    \begin{align}\label{Res23}
        \left| 2(3\beta-d\rho)\int_{\Gamma_{T}^1}\varphi z \pnu z\mathrm{d}S\mathrm{d}t \right|\leq C\int_{\Gamma_{T}^1}\varphi |z|^2\mathrm{d}S\mathrm{d}t+C\int_{\Gamma_{T}^1}\varphi |\pnu z|^2\mathrm{d}S\mathrm{d}t.
    \end{align}
The next term is non-classical and somewhat tricky because of the powers of $s$ and $\lambda$. For any $\varepsilon>0$, we have
\begin{align*}
&s^3\lambda^3\int_{\Gamma_{T}^1}\varphi^3|\pt\psi|^2(2\lambda |\pt \psi|^2-16\beta)|z|^2\mathrm{d}S\mathrm{d}t\\
=&s^3\lambda^3\int_{\Gamma^1\times\left((-T,-\varepsilon)\cup(\varepsilon,T)\right) }\varphi^3|\pt\psi|^2(8\lambda\beta^2 t^2-16\beta)|z|^2\mathrm{d}S\mathrm{d}t+2s^3\lambda^4\int_{\Gamma^1\times[-\varepsilon,\varepsilon]}\varphi^3|\pt\psi|^4|z|^2\mathrm{d}S\mathrm{d}t\\
&-64s^3\lambda^3\beta^3\int_{\Gamma^1\times[-\varepsilon,\varepsilon]}\varphi^3 t^2|z|^2\mathrm{d}S\mathrm{d}t\\
&\geq 8s^3\lambda^3\beta(\lambda \beta\varepsilon^2-2)\int_{\Gamma^1\times\left((-T,-\varepsilon)\cup(\varepsilon,T)\right) }\varphi^3|\pt\psi|^2|z|^2\mathrm{d}S\mathrm{d}t-64s^3\lambda^3\varepsilon^2\beta^3\int_{\Gamma_{T}^1}\varphi^3|z|^2\mathrm{d}S\mathrm{d}t.
\end{align*}
By choosing $\varepsilon=\dfrac{\sqrt{C_3}}{8\sqrt{2}\beta^\frac{3}{2}}$ and for $\lambda\geq \dfrac{256\beta^2}{C_3}$, the first term is nonnegative and we obtain
\begin{align}\label{Res233}
        s^3\lambda^3\int_{\Gamma_{T}^1}\varphi^3|\pt\psi|^2(2\lambda |\pt \psi|^2-16\beta)|z|^2\mathrm{d}S\mathrm{d}t\geq-\dfrac{C_3}{2}s^3\lambda^3\int_{\Gamma_{T}^1}\varphi^3|z|^2\mathrm{d}S\mathrm{d}t.
    \end{align}
    From \eqref{Res21}-\eqref{Res233}, and for large $\lambda\geq \lambda_1$ and $s\geq s_1$, we obtain
\begin{align*}
            &2\left\langle P_1 z,P_2 z\right\rangle_{L^2(\Omega_T)}+2\left\langle P_{\Gamma}z,P_{2\Gamma}z \right\rangle_{L^2(\Gamma_{T}^1)}\\
           \geq& Cs\lambda\int_{\Omega_T}\varphi\left( |\pt z|^2+|\nabla z|^2 \right)\mathrm{d}x\mathrm{d}t+C s^3\lambda^3\int_{\Omega_T}\varphi^3 |z|^2\mathrm{d}x\mathrm{d}t+[C^\prime d(\delta-d)-8\beta\delta]s\lambda \int_{\Gamma_{T}^1}\varphi|\nabla_{\Gamma}z|^2\mathrm{d}S\mathrm{d}t\\
           &\hspace{-0.5cm} +Cs^3\lambda^3\int_{\Gamma_{T}^1}\varphi^3|z|^2\mathrm{d}S\mathrm{d}t+Cs\lambda\int_{\Gamma_{T}^1}\varphi |\pt z|^2\mathrm{d}S\mathrm{d}t+Cs\lambda\int_{\Gamma_{T}^1}\varphi |\pnu z|^2\mathrm{d}S\mathrm{d}t-Cs\lambda\int_{\gamma_{T}}\varphi |\pnu z|^2\mathrm{d}S\mathrm{d}t\\
           &\hspace{-0.5cm} -\|P_{1}z\|_{L^2(\Omega_{T})}^2-\|P_{1\Gamma} z\|_{L^2(\Gamma_{T}^1)}^2-Cs\lambda\int_{\Omega}\varphi(x,T)(|\pt z(x,T)|^2+|\nabla z(x,T)|^2)\mathrm{d}x\\
           &\hspace{-0.5cm}-Cs^3\lambda^3\int_{\Omega}\varphi^3(x,T)|z(x,T)|^2\mathrm{d}x-Cs\lambda\int_{\Omega}\varphi(x,-T)(|\pt z(x,-T)|^2+|\nabla z(x,-T)|^2)\mathrm{d}x\\
            &\hspace{-0.5cm} -C s^3\lambda^3\int_{\Omega}\varphi^3(x,-T)|z(x,-T)|^2\mathrm{d}x-Cs\lambda \int_{\Gamma^1}\varphi(x,T)|\pt z(x,T)|^2\mathrm{d}S\\
            &\hspace{-0.5cm} -Cs\lambda \int_{\Gamma^1}\varphi(x,-T)|\pt z(x,-T)|^2\mathrm{d}S -Cs^3\lambda^3\int_{\Gamma^1}\varphi^3(x,T)|z(x,T)|^2\mathrm{d}S -Cs^3\lambda^3\int_{\Gamma^1}\varphi^3(x,-T)|z(x,-T)|^2\mathrm{d}S.
\end{align*}
Finally, we obtain the Carleman estimate \eqref{carleman} with $z$ instead of $y$. We can come back to the original variable $y=e^{-s\varphi} z$ by using \eqref{after:conjug} and standard estimates. This completes the proof of Theorem \ref{Thm:Carleman}.

\section{Lipschitz stability for the inverse source problem}\label{sec4}
The objective of this section is to determine two unknown forcing terms $(f,g)$ (depending on $(x,t)$) in \eqref{intro:problem:01} belonging to the following admissible set
\begin{align}\label{eqas}
\mathcal{S}(C_0) :=\left\{(f,g)\in H^1\left(0, T ; \mathcal{H}\right): \begin{array}{ll}
|\pt f(x,t)| \leq C_0 |f(x,0)|, &\text{ a.e. }(x,t)\in \Omega\times (0,T) \\
|\pt g(x,t)| \leq C_0 |g(x,0)|, &\text{ a.e. }(x,t)\in \Gamma^1\times (0,T)
\end{array}\right\},
\end{align}
where $C_0 > 0$ is fixed, from the knowledge of partial measurement $\pnu y_{|\gamma\times (0,T)}.$ In general, even uniqueness fails for general source terms $(f,g)$, which justifies the introduction of the set $\mathcal{S}(C_0)$, see \cite{Is90}.

Next, we define the minimal time
\begin{equation}\label{mt}
    T_*=\dfrac{1}{\min\left(\sqrt{\rho d},\sqrt{\frac{C^\prime d(\delta-d)}{8\delta}}\right)} \left(\underset{x\in\overline{\Omega}}{\max}\,\psi_0(x)-\underset{x\in\overline{\Omega}}{\min}\,\psi_0(x)\right)^{\frac{1}{2}},
\end{equation}
where $\rho$ is the same constant in Proposition \ref{propw}. We now state the main result of the global Lipschitz stability for the inverse source problem.
\begin{theorem}\label{thm:stab}
Let $T>T_*$ and $C_0>0$. We assume that $\delta >d$. Then there exists a positive constant $C=C(\Omega, T, C_0,\| q_{\Omega} \|_{\infty},\|q_{\Gamma}\|_{\infty})$ such that for any admissible source $(f,g)\in \mathcal{S}(C_0)$, we have
    \begin{equation}
        \| f\|_{L^2(\Omega\times (0,T))}+\| g\|_{L^2(\Gamma^1\times (0,T))}\leq C\|\pt \pnu y\|_{L^2(\gamma\times (0,T))}
    \end{equation}
for any regular solution $(y,y_\Gamma)$ of system \eqref{intro:problem:01}.
\end{theorem}

\begin{remark}\label{rmkddelta}
The assumption $\delta >d$ is essential to absorb the term with $|\nabla_\Gamma y_\Gamma|^2$ into the left-hand side of the Carleman estimate \eqref{carleman} and then in our stability proof. Otherwise, we need to add an extra measurement.
\end{remark}

\begin{proof}[Proof of Theorem \ref{thm:stab}]
We will extend a recent argument from \cite{HIY20} to our setting. Throughout the proof, $C$ will denote a generic constant
which is independent of $(y,y_\Gamma)$. Setting $(z,z_\Gamma)=(\pt y, \pt y_\Gamma)$, by \eqref{intro:problem:01}, we obtain 
\begin{align}
	\label{intro:problem:001}
	\begin{cases}
		\ptt z-d\Delta z + q_{\Omega}z=\pt f,&\text{ in }\Omega\times (0,T),\\
		\ptt z_\Gamma -\delta \Delta_\Gamma z_\Gamma +d\pnu z+ q_{\Gamma} z_\Gamma=\pt g, \quad z_\Gamma = z_{\mid_{\Gamma}}, &\text{ on }\Gamma^1\times(0,T) ,\\
		z=0,&\text{ on }\Gamma^0\times(0.T),\\
		(z(\cdot,0),z_\Gamma (\cdot,0))=(0,0), \quad(\pt z(\cdot,0),\pt z_\Gamma (\cdot,0))=(f(\cdot,0),g(\cdot,0)), &\text{ in }\Omega\times \Gamma^1.
	\end{cases}
\end{align}
To apply the Carleman estimate \eqref{carleman}, we extend the function $(z,z_\Gamma)$ to $(-T,T)$ by the odd extension $(z(\cdot,-t),z_{\Gamma}(\cdot,-t))=-(z(\cdot,t),z_{\Gamma}(\cdot,t))$ for $0<t<T$. Similarly for $\partial_t f$ and $\partial_{t}g$ in $(-T,0)$. Then, we extend the system \eqref{intro:problem:001} to $(-T,T)$
\begin{align*}
	\begin{cases}
		\ptt z-d\Delta z + q_{\Omega}z=\pt f,&\text{ in }\Omega_T,\\
		\ptt z_\Gamma -\delta \Delta_\Gamma z_\Gamma +d\pnu z+ q_{\Gamma} z_\Gamma=\pt g, \quad z_\Gamma = z_{\mid_{\Gamma}}, &\text{ on }\Gamma_T^1 ,\\
		z=0,&\text{ on }\Gamma^0_T,\\
		(z(\cdot,0),z_\Gamma (\cdot,0))=(0,0), \quad(\pt z(\cdot,0),\pt z_\Gamma (\cdot,0))=(f(\cdot,0),g(\cdot,0)), &\text{ in }\Omega\times \Gamma^1.
	\end{cases}
\end{align*}
Next, we apply Theorem \ref{Thm:Carleman} to $(z,z_\Gamma)$. Using the fact that $(z(\cdot,-t),z_{\Gamma}(\cdot,-t))=-(z(\cdot,t),z_{\Gamma}(\cdot,t))$ and $\varphi(\cdot,-t)=\varphi(\cdot,t)$ in $\overline{\Omega}$, we obtain
 \begin{align}\label{PrIn00}
     &\int_{0}^T\int_{\Omega} {e^{2s\varphi} \left(s^3\lambda^3\varphi^3|z|^2 +s\lambda \varphi |\nabla z|^2 +s\lambda \varphi |\pt z|^2 \right) }\mathrm{d}x \mathrm{d}t \notag\\
	&+\int_{0}^T\int_{\Gamma^1} {e^{2s\varphi} \left(s^3\lambda^3\varphi^3 |z_\Gamma|^2 + s\lambda \varphi|\pnu z|^2 + s\lambda \varphi|\pt z_\Gamma|^2 + [C^\prime d(\delta-d)-8\beta\delta]s\lambda \varphi |\nabla_\Gamma z_\Gamma|^2 \right)}\mathrm{d}S\mathrm{d}t\notag\\
	\leq & C\int_{0}^T\int_{\Omega}{e^{2s\varphi}|\pt f|^2}\mathrm{d}x\mathrm{d}t
	+ C \int_{0}^T\int_{\Gamma^{1}}{e^{2s\varphi} |\pt g|^2}\mathrm{d}S\mathrm{d}t+C s\lambda \int_{0}^T \int_{\gamma}{e^{2s\varphi} \varphi|\pnu z|^2}\mathrm{d}S\mathrm{d}t\notag\\
	&+Cs\lambda\int_{\Omega}\mathrm{e}^{2s\varphi(x,T)}\varphi(x,T)(|\pt z(x,T)|^2+|\nabla z(x,T)|^2)\mathrm{d}x\notag\\
    &+Cs^3\lambda^3\int_{\Omega}e^{2s\varphi(x,T)}\varphi^3(x,T)|z(x,T)|^2\mathrm{d}x+Cs\lambda \int_{\Gamma^1}e^{2s\varphi(x,T)}\varphi(x,T)|\pt z_{\Gamma}(x,T)|^2\mathrm{d}S\notag\\
    &+Cs^3\lambda^3\int_{\Gamma^1}e^{2s\varphi(x,T)}\varphi^3(x,T)|z_\Gamma(x,T)|^2\mathrm{d}S.
\end{align} 
Now, we estimate the term $\displaystyle\int_{\Omega}|\pt z(x,0)|^2e^{2s\varphi(x,0)}\mathrm{d}x+\int_{\Gamma^1}|\pt z_{\Gamma}(x,0)|^2e^{2s\varphi(x,0)}\mathrm{d}S$ as follows
\begin{align*}
    & \int_{\Omega}|\pt z(x,0)|^2e^{2s\varphi(x,0)}\mathrm{d}x+\int_{\Gamma^1}|\pt z_{\Gamma}(x,0)|^2e^{2s\varphi(x,0)}\mathrm{d}S\\
    =&-2\int_{0}^T \int_{\Omega}\left[ s\pt\varphi |\pt z|^2+\pt z (d\Delta z -q_{\Omega}z+\pt f)\right]e^{2s\varphi}\mathrm{d}x\mathrm{d}t+\int_{\Omega}e^{2s\varphi(x,T)}|\pt z(x,T)|^2\mathrm{d}x\\
    & \hspace{-0.5cm} -2\int_{0}^T \int_{\Gamma^1}\left[ s\pt\varphi |\pt z_\Gamma|^2+\pt z_\Gamma (\delta\Delta_{\Gamma}z_{\Gamma}-d\pnu z-q_{\Gamma}z_\Gamma+\pt g)\right]e^{2s\varphi}\mathrm{d}S\mathrm{d}t+\int_{\Gamma^1}e^{2s\varphi(x,T)}|\pt z_\Gamma(x,T)|^2\mathrm{d}S.
 \end{align*}
By using the Green formula with \eqref{estimate:der:weights:01}, the fact that $z=0$ on $\Gamma^0$, and $(z(\cdot,0),z_{\Gamma}(\cdot,0))=(0,0)$ in $\Omega \times \Gamma^1$, we obtain
\begin{align}\label{PbIn01}
  \begin{split}
    & \int_{\Omega}|\pt z(x,0)|^2e^{2s\varphi(x,0)}\mathrm{d}x+\int_{\Gamma^1}|\pt z_{\Gamma}(x,0)|^2e^{2s\varphi(x,0)}\mathrm{d}S\\
    =&-2s\lambda\int_{0}^T\int_\Omega \pt\psi \varphi |\pt z|^2 e^{2s\varphi}\mathrm{d}x \mathrm{d}t + d\int_{\Omega} |\nabla z(x,T)|^2 e^{2s\varphi(x,T)}\mathrm{d}x\mathrm{d}t-2d s\lambda\int_{0}^T\int_\Omega \pt\psi \varphi |\nabla z|^2e^{2s\varphi}\mathrm{d}x\mathrm{d}t\\
    &+4s\lambda d\int_{0}^T \int_\Omega \varphi\pt z \nabla \psi\cdot \nabla ze^{2s\varphi}\mathrm{d}x\mathrm{d}t+2\int_0^T \int_\Omega q_\Omega z\pt z e^{2s\varphi}\mathrm{d}x\mathrm{d}t-2\int_0^T \int_\Omega \pt z \pt f e^{2s\varphi}\mathrm{d}x\mathrm{d}t\\
    &+\int_\Omega |\pt z(x,T)|^2 e^{2s\varphi(x,T)}\mathrm{d}x-2s\lambda\int_{0}^T\int_{\Gamma^1} \pt\psi\varphi |\pt z|^2 e^{2s\varphi}\mathrm{d}S\mathrm{d}t+\delta\int_{\Gamma^1}|\nabla_{\Gamma }z_\Gamma(x,T)|^2e^{2s\varphi(x,T)}\mathrm{d}S\\
    &-2\delta s \lambda\int_{0}^T\int_{\Gamma^1} \pt \psi\varphi |\nabla_{\Gamma} z_{\Gamma}|^2e^{2s\varphi}\mathrm{d}S\mathrm{d}t+4s\lambda\delta\int_{0}^T \int_{\Gamma^1} \varphi\pt z \left\langle\nabla_\Gamma \psi, \nabla_\Gamma z_{\Gamma}\right\rangle_{\Gamma}e^{2s\varphi}\mathrm{d}S\mathrm{d}t\\
    &+2\int_0^T \int_{\Gamma^1} q_\Gamma z_\Gamma\pt z_\Gamma e^{2s\varphi}\mathrm{d}S\mathrm{d}t-2\int_0^T \int_{\Gamma^1} \pt z_\Gamma \pt g e^{2s\varphi}\mathrm{d}S\mathrm{d}t+\int_{\Gamma^1} |\pt z_\Gamma(x,T)|^2 e^{2s\varphi(x,T)}\mathrm{d}S.
    \end{split}
\end{align}
On the other hand, applying Young's inequality, we obtain
\begin{align}\label{Youn01}
 \begin{array}{ll}
    &|\pt z \nabla\psi\cdot\nabla z|\leq C(|\pt z|^2+|\nabla z|^2), \hspace{2.1cm} |\pt z(q_{\Omega}z-\pt f)|\leq C(|\pt z|^2+|z|^2+|\pt f|^2),\\
    &|\pt z_{\Gamma}\left\langle \nabla_{\Gamma}\psi,\nabla_{\Gamma}z_{\Gamma}\right\rangle_{\Gamma}|\leq C(|\pt z|^2+|\nabla_{\Gamma}z|^2), \qquad |\pt z_{\Gamma}(q_{\Gamma}z_{\Gamma}-\pt g)|\leq C(|\pt z|^2+|z|^2+|\pt g|^2).
     \end{array}
\end{align}
By \eqref{PbIn01} and \eqref{Youn01}, we deduce
\begin{align}\label{PbIn02}
    & \int_{\Omega}|\pt z(x,0)|^2e^{2s\varphi(x,0)}\mathrm{d}x+\int_{\Gamma^1}|\pt z_{\Gamma}(x,0)|^2e^{2s\varphi(x,0)}\mathrm{d}S\notag\\
    \leq&C\int _{0}^T\int_{\Omega}|\pt f|^2 e^{2s\varphi}\mathrm{d}x\mathrm{d}t+C\int_{0}^T\int_{\Gamma^1}|\pt g|^2 e^{2s\varphi}\mathrm{d}S\mathrm{d}t+C\int_{0}^T\int_\Omega |z|^2 e^{2s\varphi}\mathrm{d}x\mathrm{d}t\notag\\
    &+C\int_{0}^T\int_{\Gamma^1} |z_{\Gamma}|^2 e^{2s\varphi}\mathrm{d}S\mathrm{d}t+Cs\lambda \int_{0}^T\int_\Omega (|\pt z|^2+|\nabla z|^2) e^{2s\varphi}\mathrm{d}x\mathrm{d}t\notag\\
    &+Cs\lambda \int_{0}^T\int_{\Gamma^1} (|\pt z_{\Gamma}|^2+|\nabla_{\Gamma} z_{\Gamma}|^2) e^{2s\varphi}\mathrm{d}S\mathrm{d}t+C\int_\Omega (|\pt z(x,T)|^2+|\nabla z(x,T)|^2) e^{2s\varphi(x,T)}\mathrm{d}x\notag\\
    &+C\int_{\Gamma^1} (|\pt z_{\Gamma}(x,T)|^2+|\nabla_{\Gamma} z_{\Gamma}(x,T)|^2) e^{2s\varphi(x,T)}\mathrm{d}S.
   \end{align}
Since $(\pt z(\cdot,0),\pt z_\Gamma(\cdot,0))=(f(\cdot,0),g(\cdot,0))$ in $\Omega\times \Gamma$ and $(f,g)\in \mathcal{S}(C_0)$, then 
\begin{align}\label{PbIn03}
    & \int_{\Omega}|f(x,0)|^2e^{2s\varphi(x,0)}\mathrm{d}x+\int_{\Gamma^1}|g(x,0)|^2e^{2s\varphi(x,0)}\mathrm{d}S\notag\\
    \leq&C\int _{0}^T\int_{\Omega}|f(x,0)|^2 e^{2s\varphi}\mathrm{d}x\mathrm{d}t+C\int_{0}^T\int_{\Gamma^1}|g(x,0)|^2 e^{2s\varphi}\mathrm{d}S\mathrm{d}t+C\int_{0}^T\int_\Omega |z|^2 e^{2s\varphi}\mathrm{d}x\mathrm{d}t \notag\\
    &+C\int_{0}^T\int_{\Gamma^1} |z_{\Gamma}|^2 e^{2s\varphi}\mathrm{d}S\mathrm{d}t+Cs\lambda \int_{0}^T\int_\Omega (|\pt z|^2+|\nabla z|^2) e^{2s\varphi}\mathrm{d}x\mathrm{d}t\notag\\
    &+Cs\lambda \int_{0}^T\int_{\Gamma^1} (|\pt z_{\Gamma}|^2+|\nabla_{\Gamma} z_{\Gamma}|^2) e^{2s\varphi}\mathrm{d}S\mathrm{d}t+C\int_\Omega (|\pt z(x,T)|^2+|\nabla z(x,T)|^2) e^{2s\varphi(x,T)}\mathrm{d}x \notag\\
    &+C\int_{\Gamma^1} (|\pt z_{\Gamma}(x,T)|^2+|\nabla_{\Gamma} z_{\Gamma}(x,T)|^2) e^{2s\varphi(x,T)}\mathrm{d}S.
   \end{align}
   Combining \eqref{PrIn00} and \eqref{PbIn03}, we obtain
   \begin{align}\label{PbIn04}
    & \int_{\Omega}|f(x,0)|^2e^{2s\varphi(x,0)}\mathrm{d}x+\int_{\Gamma^1}|g(x,0)|^2e^{2s\varphi(x,0)}\mathrm{d}S \notag\\
    \leq&C\int _{0}^T\int_{\Omega}|f(x,0)|^2 e^{2s\varphi}\mathrm{d}x\mathrm{d}t+C\int_{0}^T\int_{\Gamma^1}|g(x,0)|^2 e^{2s\varphi}\mathrm{d}S\mathrm{d}t+Cs\lambda \int_{0}^T\int_{\gamma}\varphi |\pt \pnu y|^2e^{2s\varphi}\mathrm{d}S\mathrm{d}t \notag\\
    &\hspace{-0.3cm} +Cs\lambda \int_\Omega e^{2s\varphi(x,T)} \varphi(x,T)\left( |\pt z(x,T)|^2+|\nabla z(x,T)|^2 \right)\mathrm{d}x +Cs^3\lambda^3 \int_\Omega e^{2s\varphi(x,T)}\varphi^3(x,T)|z(x,T)|^2  \mathrm{d}x\notag\\
    & \hspace{-0.3cm}+Cs\lambda \int_{\Gamma^1} e^{2s\varphi(x,T)} \varphi(x,T) (|\pt z_{\Gamma}(x,T)|^2+|\nabla_{\Gamma} z_{\Gamma}(x,T)|^2)\mathrm{d}S  +Cs^3\lambda^3 \int_{\Gamma^1} e^{2s\varphi(x,T)}\varphi^3(x,T)|z(x,T)|^2  \mathrm{d}S
   \end{align}
for sufficiently large $\lambda$ and $s$.

On the other hand, we have $e^{-2s(\varphi(x,0)-\varphi(x,t))}=e^{-2se^{\lambda(\psi_0(x)+C_1)}(1-e^{-\lambda\beta t^2})}$. 
Since $e^{\lambda(\psi_0(x)+C_1)}>1$, then $e^{-2s(\varphi(x,0)-\varphi(x,t))}\leq e^{-2s(1-e^{-\lambda\beta t^2})}$.
Thus, 
\begin{align*}
    & C\int _{0}^T\int_{\Omega}|f(x,0)|^2 e^{2s\varphi}\mathrm{d}x\mathrm{d}t+C\int_{0}^T\int_{\Gamma^1}|g(x,0)|^2 e^{2s\varphi}\mathrm{d}S\mathrm{d}t\\
    &\leq C \int_0^T e^{-2s(1-e^{-\lambda\beta t^2})}\mathrm{d}t\left( \int_{\Omega}|f(x,0)|^2 e^{2s\varphi(x,0)}\mathrm{d}x+\int_{\Gamma^1}|g(x,0)|^2 e^{2s\varphi(x,0)}\mathrm{d}S \right).
\end{align*}
Let us fix $0< t\leq T$. Since $e^{-2s(1-e^{-\lambda\beta t^2})}\rightarrow 0$ as $s\to \infty$, we can apply the dominated convergence theorem, leading to
\begin{align}\label{PrIn05}
    \begin{split}
        &C\int _{0}^T\int_{\Omega}|f(x,0)|^2 e^{2s\varphi}\mathrm{d}x\mathrm{d}t+C\int_{0}^T\int_{\Gamma^1}|g(x,0)|^2 e^{2s\varphi}\mathrm{d}S\mathrm{d}t\\
        \leq& \dfrac{1}{2}\left( \int_{\Omega}|f(x,0)|^2 e^{2s\varphi(x,0)}\mathrm{d}x+\int_{\Gamma^1}|g(x,0)|^2 e^{2s\varphi(x,0)}\mathrm{d}S \right)
    \end{split}
\end{align}
for all large $\lambda$. Now, summing up inequalities \eqref{PbIn04} and \eqref{PrIn05}, we conclude that
\begin{align}\label{PbIn06}
& \int_{\Omega}|f(x,0)|^2e^{2s\varphi(x,0)}\mathrm{d}x+\int_{\Gamma^1}|g(x,0)|^2e^{2s\varphi(x,0)}\mathrm{d}S \notag\\
\leq &Cs\lambda \int_{0}^T\int_{\gamma}\varphi |\pt \pnu y|^2e^{2s\varphi}\mathrm{d}S\mathrm{d}t+Cs\lambda \int_\Omega e^{2s\varphi(x,T)} \varphi(x,T)\left( |\pt z(x,T)|^2+ |\nabla z(x,T)|^2 \right)\mathrm{d}x \notag\\
& \hspace{-0.5cm} +Cs^3\lambda^3 \int_\Omega e^{2s\varphi(x,T)}\varphi^3(x,T)|z(x,T)|^2  \mathrm{d}x +Cs^3\lambda^3 \int_{\Gamma^1} e^{2s\varphi(x,T)}\varphi^3(x,T)|z(x,T)|^2  \mathrm{d}S \notag\\
&\hspace{-0.5cm}+Cs\lambda \int_{\Gamma^1} e^{2s\varphi(x,T)} \varphi(x,T) (|\pt z_{\Gamma}(x,T)|^2+|\nabla_{\Gamma} z_{\Gamma}(x,T)|^2)\mathrm{d}S
\end{align}
for all large $\lambda$. Now, we prove that 
   \begin{align}\label{PbIn07}
       \begin{split}
           &\|\nabla z(\cdot,T)\|_{L^2(\Omega)}^2+\|\pt z (\cdot,T)\|_{L^2(\Omega)}^2+ \|\nabla_{\Gamma} z_{\Gamma}(\cdot,T)\|_{L^2\left(\Gamma^1\right)}^2+\|\pt z_{\Gamma} (\cdot,T)\|_{L^2(\Gamma^1)}^2\\
           \leq &C_{T}\left( \|f(\cdot,0)\|_{L^2(\Omega)}^2+\|g(\cdot,0)\|_{L^2\left(\Gamma^1\right)}^2 \right).
       \end{split}
   \end{align}
We set
\begin{align*}
\begin{split}
              I^2(\tau):=&\|\sqrt{d}\nabla z(\cdot,\tau)\|_{L^2(\Omega)}^2+\|\partial_\tau z (\cdot,\tau)\|_{L^2(\Omega)}^2+ \|\sqrt{\delta}\nabla_{\Gamma} z_{\Gamma}(\cdot,\tau)\|_{L^2\left(\Gamma^1\right)}^2+\|\partial_\tau z_{\Gamma} (\cdot,\tau)\|_{L^2(\Gamma^1)}^2\\           &+\|q_{\Omega}z(\cdot,\tau)\|_{L^2(\Omega)}^2+\|q_{\Gamma}z_{\Gamma}(\cdot,\tau)\|_{L^2(\Gamma^1)}^2, \quad\quad 0\leq \tau<T.
    \end{split}
   \end{align*}
We first compute $\dfrac{\mathrm{d}}{\mathrm{d}\tau}I^2(\tau)$. By using Green's formula and the fact that $z=0$ on $\Gamma^0$, we obtain
   \begin{align*}
       \begin{split}
           \dfrac{\mathrm{d}}{\mathrm{d}\tau}I^2(\tau)=&2\int_{\Omega}\left(\partial_\tau^2 z-d\Delta z +q_{\Omega}z \right)\partial_\tau z\mathrm{d}x+2\int_{\Gamma^1}\left( \partial_\tau^2 z_{\Gamma}-\delta\Delta_{\Gamma}z_{\Gamma} +d\pnu z +q_{\Gamma}z_{\Gamma}\right)\partial_\tau z_{\Gamma}\mathrm{d}S\\
           =&2\int_{\Omega}\partial_\tau z \partial_\tau f\mathrm{d}x+2\int_{\Gamma^1}\partial_\tau z_{\Gamma}\partial_\tau g\mathrm{d}S.
       \end{split}
   \end{align*}
By Cauchy-Schwarz inequality, we infer
   \begin{equation*}
       \dfrac{\mathrm{d}}{\mathrm{d}\tau}I^2(\tau)\leq2\|\partial_\tau z\|_{L^2(\Omega)}\|\partial_\tau f(\cdot,\tau)\|_{L^2(\Omega)}+2\|\partial_\tau z_{\Gamma}\|_{L^2\left(\Gamma^1\right)}\|\partial_\tau g(\cdot,\tau)\|_{L^2\left(\Gamma^1\right)}.
   \end{equation*}
Using the definition of $I^2(\tau)$, it is clear that $\|\partial_\tau z\|_{L^2(\Omega)}\leq I(\tau) \text{ and } \|\partial_\tau z_{\Gamma}\|_{L^2\left(\Gamma^1\right)}\leq I(\tau).$
Hence
   \begin{equation*}
       2 I(\tau)\dfrac{\mathrm{d}}{\mathrm{d}\tau}I(\tau)\leq 2 I(\tau)\left(\|\partial_\tau f(\cdot,\tau)\|_{L^2(\Omega)}+\|\partial_\tau g(\cdot,\tau)\|_{L^2\left(\Gamma^1\right)}\right).
   \end{equation*}
Therefore, since $I(\tau)\geq 0$, we infer $\dfrac{\mathrm{d}}{\mathrm{d}\tau}I(\tau)\leq \|\partial_\tau f(\cdot,\tau)\|_{L^2(\Omega)}+\|\partial_\tau g(\cdot,\tau)\|_{L^2\left(\Gamma^1\right)}.$
Integrating over $(0,t)$, we obtain
   \begin{equation*}
       I(t)\leq I(0)+\int_{0}^t \|\partial_\tau f(\cdot,\tau)\|_{L^2(\Omega)}\mathrm{d}\tau+\int_{0}^t\| \partial_\tau g(\cdot,\tau) \|_{L^2\left(\Gamma^1\right)}\mathrm{d}\tau.
       \end{equation*}
Then Young's inequality  yields
\begin{equation*}
        I^2(t)\leq 3I^2(0)+3\left( \int_{0}^t\|\partial_\tau f(\cdot,\tau)\|_{L^2(\Omega)}\mathrm{d}\tau \right)^2+3\left( \int_{0}^t \|\partial_\tau g(\cdot,\tau)\|_{L^2\left(\Gamma^1\right)} \mathrm{d}\tau\right)^2.
\end{equation*}
Since $(z(\cdot,0),z_{\Gamma}(\cdot,0))=(0,0)$ in $\Omega\times \Gamma^1$, then $I^2(0)=\| \partial_\tau z(\cdot,0) \|_{L^2(\Omega)}^2+\| \partial_\tau z_{\Gamma}(\cdot,0) \|_{L^2(\Gamma^1)}^2.$
Thus,
\begin{align*}
\begin{split}
I^2(t) &\leq 3\|f(\cdot,0)\|_{L^2(\Omega)}^2+3\|g(\cdot,0)\|_{L^2\left(\Gamma^1\right)}^2+3\int_{0}^T\|\partial_\tau f(\cdot,\tau)\|_{L^2(\Omega)}^2\mathrm{d}\tau+3 \int_{0}^T \|\partial_\tau g(\cdot,\tau)\|_{L^2\left(\Gamma^1\right)}^2\mathrm{d}\tau.
\end{split}
\end{align*}
Since $(f,g)\in \mathcal{S}(C_0)$, then $I^2(t)\leq 3(1+TC_0)\left(\|f(\cdot,0)\|_{L^2(\Omega)}^2+\|g(\cdot,0)\|_{L^2\left(\Gamma^1\right)}^2\right).$ 
This leads to the desired inequality \eqref{PbIn07}.

On the other hand, $z(\cdot,T)\in H_{0,\Gamma^0}^1(\Omega):=\left\lbrace y\in H^1(\Omega) \colon \;\;y_{\mid_{\Gamma^0}}=0 \right\rbrace$ and $\Gamma^0$ has a positive surface measure, then by the Poincaré inequality (see \cite[p.~177]{Zi89}), we obtain 
$
    \|z(\cdot,T)\|_{L^2(\Omega)}\leq C\|\nabla z(\cdot,T)\|_{L^2(\Omega)}.
$
Then along with the following trace estimate
\begin{equation*}
    \|z_{\Gamma}(\cdot,T)\|_{L^2\left(\Gamma^1\right)}^2\leq C\left(\|z(\cdot,T)\|_{L^2(\Omega)}^2+\|\nabla z(\cdot,T)\|_{L^2(\Omega)}^2\right),
\end{equation*}
we obtain
$
     \|z_{\Gamma}(\cdot,T)\|_{L^2\left(\Gamma^1\right)}^2\leq C\|\nabla z(\cdot,T)\|_{L^2(\Omega)}^2.
$
Now the previous estimates yield
\begin{align*}
    \begin{split}
        &\|z(\cdot,T)\|_{L^2(\Omega)}^2+\|\pt z(\cdot,T)\|_{L^2(\Omega)}^2 +\|\nabla z(\cdot,T)\|_{L^2(\Omega)}^2+\|z_{\Gamma}(\cdot,T)\|_{L^2\left(\Gamma^1\right)}^2+\|\pt z_{\Gamma}(\cdot,T)\|_{L^2\left(\Gamma^1\right)}^2+\|\nabla_{\Gamma}z(\cdot,T)\|_{L^2(\Gamma^1)}^2\\
        \leq &C\left(\|f(\cdot,0)\|_{L^2(\Omega)}^2+\|g(\cdot,0)\|_{L^2\left(\Gamma^1\right)}^2\right).
    \end{split}
\end{align*}
Therefore
\begin{align}\label{PbIn08}
    &Cs\lambda \int_\Omega e^{2s\varphi(x,T)} \varphi(x,T)\left( |\pt z(x,T)|^2+|\nabla z|^2 \right)\mathrm{d}x+Cs^3\lambda^3 \int_\Omega e^{2s\varphi(x,T)}\varphi^3(x,T)|z(x,T)|^2  \mathrm{d}x \notag\\
    &+Cs\lambda \int_{\Gamma^1} e^{2s\varphi(x,T)} \varphi(x,T) |\pt z_{\Gamma}(x,T)|^2\mathrm{d}S+Cs^3\lambda^3 \int_{\Gamma^1} e^{2s\varphi(x,T)}\varphi^3(x,T)|z(x,T)|^2  \mathrm{d}S \notag\\
    &\leq Cs^3\lambda^3e^{3\lambda(d_1^2-\beta T^2+C_1)} e^{2se^{3\lambda(d_1^2-\beta T^2+C_1)}}\left(\|f(\cdot,0)\|_{L^2(\Omega)}^2+\|g(\cdot,0)\|_{L^2\left(\Gamma^1\right)}^2\right),
\end{align}
where $d_1:=\underset{x\in\overline{\Omega}}{\max}\, \mu(x)$. Combining \eqref{PbIn06} and \eqref{PbIn08}, we obtain
\begin{align}\label{PbIn09}
  \begin{split}
    & \int_{\Omega}|f(x,0)|^2e^{2s\varphi(x,0)}\mathrm{d}x+\int_{\Gamma^1}|g(x,0)|^2e^{2s\varphi(x,0)}\mathrm{d}S\\
    &\leq Cs\lambda \int_{0}^T\int_{\gamma}\varphi |\pt \pnu y|^2e^{2s\varphi}\mathrm{d}S\mathrm{d}t +Cs^3\lambda^3e^{3\lambda(d_1^2-\beta T^2+C_1)} e^{2se^{\lambda(d_1^2-\beta T^2+C_1)}}\left(\|f(\cdot,0)\|_{L^2(\Omega)}^2+\|g(\cdot,0)\|_{L^2\left(\Gamma^1\right)}^2\right).
\end{split}
\end{align}
Moreover,
\begin{align}\label{PbIn10}
    \begin{split}
         & \int_{\Omega}|f(x,0)|^2e^{2s\varphi(x,0)}\mathrm{d}x+\int_{\Gamma^1}|g(x,0)|^2e^{2s\varphi(x,0)}\mathrm{d}S\\
         &=\int_{\Omega}|f(x,0)|^2e^{2se^{\lambda(\psi_0(x)+C_1)}}\mathrm{d}x+\int_{\Gamma^1}|g(x,0)|^2e^{2se^{\lambda(\psi_0(x)+C_1)}}\mathrm{d}S\\
         &\geq e^{2se^{\lambda(d_0^2+C_1)}}\left(\|f(\cdot,0)\|_{L^2(\Omega)}^2+\|g(\cdot,0)\|_{L^2\left(\Gamma^1\right)}^2\right),
    \end{split}
\end{align}
where $d_0:=\underset{x\in\overline{\Omega}}{\min}\,\mu(x)$. By \eqref{PbIn09} and \eqref{PbIn10}, we obtain
\begin{align}\label{PbIn11}
\begin{split}
&\|f(\cdot,0)\|_{L^2(\Omega)}^2+\|g(\cdot,0)\|_{L^2\left(\Gamma^1\right)}^2\\
\leq &C s\lambda e^{-2se^{\lambda(d_0^2+C_1)}} \int_{0}^T\int_{\gamma}\varphi |\pt \pnu y|^2e^{2s\varphi}\mathrm{d}S\mathrm{d}t\\
&\hspace{-0.4cm} +Cs^3\lambda^3e^{3\lambda(d_1^2-\beta T^2+C_1)} e^{2s\left(e^{\lambda(d_1^2-\beta T^2+C_1)}-e^{\lambda(d_0^2+C_1)}\right)}\left(\|f(\cdot,0)\|_{L^2(\Omega)}^2+\|g(\cdot,0)\|_{L^2\left(\Gamma^1\right)}^2\right).
\end{split}
\end{align}
On one hand, we have 
\begin{equation*}
        C s\lambda e^{-2se^{\lambda(d_0^2+C_1)}} \int_{0}^T\int_{\gamma}\varphi |\pt \pnu y|^2e^{2s\varphi}\mathrm{d}S\mathrm{d}t\leq C s\lambda e^{\lambda(d_{1}^2+C_1)}e^{-2se^{\lambda(d_0^2+C_1)}} \int_{0}^T\int_{\gamma} |\pt \pnu y|^2e^{2s\varphi}\mathrm{d}S\mathrm{d}t.
\end{equation*}
Since $\lambda e^{\lambda(d_{1}^2+C_1)}e^{-2se^{\lambda(d_0^2+C_1)}}\rightarrow 0$ as $\lambda\to \infty$, for $\lambda$ large enough, we obtain $\lambda e^{\lambda(d_{1}^2+C_1)}e^{-2se^{\lambda(d_0^2+C_1)}}\leq C.$
Therefore 
\begin{align}\label{PbIn12}
       C s\lambda e^{-2se^{\lambda(d_0^2+C_1)}} \int_{0}^T\int_{\gamma}\varphi |\pt \pnu y|^2e^{2s\varphi}\mathrm{d}S\mathrm{d}t
        \leq Ce^{Cs} \int_{0}^T\int_{\gamma} |\pt \pnu y|^2\mathrm{d}S\mathrm{d}t.
\end{align}
On the other hand, we have $e^{\lambda(d_1^2-\beta T^2+C_1)}-e^{\lambda(d_0^2+C_1)}<0.$
Indeed, we have $e^{\lambda(d_1^2-\beta T^2+C_1)}-e^{\lambda(d_0^2+C_1)}=e^{\lambda(d_0^2+C_1)}\left(e^{\lambda(d_1^2-d_0^2-\beta T^2)}-1\right).$ Since $T>\dfrac{1}{\min\left(\sqrt{\rho d},\sqrt{\frac{C^\prime d(\delta-d)}{8\delta}}\right)}\sqrt{d_1^2-d_0^2}=:T_*$, we can choose $\beta\in \left(\frac{d_1^2-d_0^2}{T^2}, \min\left(\rho d,\frac{C^\prime d(\delta-d)}{8\delta}\right)\right)$ (see Remark \ref{rmkbeta}) so that $d_1^2-d_0^2-\beta T^2<0$, and then $e^{\lambda(d_1^2-d_0^2-\beta T^2)}<1.$ Then, for a fixed $\lambda$, we obtain
\begin{equation*}
    \underset{s\rightarrow\infty}\lim s^3\lambda^3e^{3\lambda(d_1^2-\beta T^2+C_1)} e^{2s\left(e^{\lambda(d_1^2-\beta T^2+C_1)}-e^{\lambda(d_0^2+C_1)}\right)}=0.
\end{equation*}
By taking $s$ sufficiently large, we obtain
\begin{align}\label{PbIn13}
    \begin{split}
        &Cs^3\lambda^3e^{3\lambda(d_1^2-\beta T^2+C_1)} e^{2s\left(e^{\lambda(d_1^2-\beta T^2+C_1)}-e^{\lambda(d_0^2+C_1)}\right)}\left(\|f(\cdot,0)\|_{L^2(\Omega)}^2+\|g(\cdot,0)\|_{L^2\left(\Gamma^1\right)}^2\right)\\
        \leq& \dfrac{1}{2}\left(\|f(\cdot,0)\|_{L^2(\Omega)}^2+\|g(\cdot,0)\|_{L^2\left(\Gamma^1\right)}^2\right).
    \end{split}
\end{align}
Consequently \eqref{PbIn11}-\eqref{PbIn13} yield
\begin{equation}\label{PbIn14}
\|f(\cdot,0)\|_{L^2(\Omega)}^2+\|g(\cdot,0)\|_{L^2\left(\Gamma^1\right)}^2\leq Ce^{Cs} \int_{0}^T\int_{\gamma} |\pt \pnu y|^2\mathrm{d}S\mathrm{d}t.
\end{equation}
For $t\in (0,T)$, we have  $\displaystyle f(x,t)=f(x,0)+\int_{0}^t \partial_{\tau}f(x,\tau)\mathrm{d}\tau\; \text{and} \; g(x,t)=g(x,0)+\int_{0}^t \partial_{\tau}g(x,\tau)\mathrm{d}\tau.$ Since $(f,g)\in \mathcal{S}(C_0)$, then
\begin{align}\label{PbIn15}
\begin{split}
        &|f(x,t)|\leq |f(x,0)|+\int_{0}^T|\partial_{\tau}f(x,\tau)|\mathrm{d}\tau\leq (1+C_0 T)|f(x,0)|,\\
        &|g(x,t)|\leq |g(x,0)|+\int_{0}^T|\partial_{\tau}g(x,\tau)|\mathrm{d}\tau\leq (1+C_0 T)|g(x,0)|.
        \end{split}
\end{align}
Combining \eqref{PbIn14} and \eqref{PbIn15}, we obtain
\begin{equation*}
\|f\|_{L^2(\Omega\times(0,T))}^2+\|g\|_{L^2(\Gamma^1\times (0,T))}^2\leq CT(1+C_0 T)e^{Cs} \int_{0}^T\int_{\gamma} |\pt \pnu y|^2\mathrm{d}S\mathrm{d}t.
\end{equation*}
Thus, by fixing $s\geq s_1$, the proof is achieved.
\end{proof}

\section{Boundary controllability of the wave system}\label{sec5}
In this section, we briefly discuss a second application of our Carleman estimate (Theorem \ref{Thm:Carleman}) to the exact boundary controllability of the system
\begin{align}
	\label{ctrlpb}
	\begin{cases}
		\ptt y-d\Delta y + q_{\Omega}(x,t) y=0, &\text{ in } \Omega \times (0,T),\\
		\ptt y_\Gamma -\delta \Delta_\Gamma y_\Gamma + d\pnu y + q_{\Gamma}(x,t) y_\Gamma=0, \quad y_\Gamma = y_{\mid_{\Gamma}}, &\text{ on }\Gamma^1\times (0,T) ,\\
		y=\mathds{1}_\gamma v,&\text{ on }\Gamma^0\times (0,T),\\
		(y(\cdot,0),y_\Gamma (\cdot,0))=(y_0,y_{0,\Gamma}), \quad(\pt y(\cdot,0),\pt y_\Gamma (\cdot,0))=(y_1,y_{1,\Gamma}), &\text{ in }\Omega\times \Gamma^1,
	\end{cases}
\end{align}
with one control force $v\in L^2(\gamma\times (0,T))$, where $q_{\Omega} \in L^{\infty}(\Omega\times (0,T))$, $q_{\Gamma} \in L^{\infty}\left(\Gamma^1\times (0,T)\right)$, and $\gamma$ is the control region defined in \eqref{Inclu}. The solution of \eqref{ctrlpb} with less regular initial data should be understood in the transposition sense.

Simultaneously with the system \eqref{ctrlpb}, we consider its adjoint backward system
\begin{align}
	\label{adjpb}
	\begin{cases}
		\ptt z-d\Delta z + q_{\Omega}(x,t) z=0, &\text{ in } \Omega\times (0,T),\\
		\ptt z_\Gamma -\delta \Delta_\Gamma z_\Gamma + d\pnu z + q_{\Gamma}(x,t) z_\Gamma=0, \quad z_\Gamma = z_{\mid_{\Gamma}}, &\text{ on }\Gamma^1\times (0,T) ,\\
		z=0,&\text{ on }\Gamma^0\times (0,T),\\
		(z(\cdot,T),z_\Gamma (\cdot,T))=(z_0,z_{0,\Gamma}), \quad(\pt z(\cdot,T),\pt z_\Gamma (\cdot,T))=(z_1,z_{1,\Gamma}), &\text{ in }\Omega\times \Gamma^1.
	\end{cases}
\end{align}

By a standard duality argument (see e.g. \cite[Corollary 1.2]{Wa18}), the exact controllability of the system \eqref{ctrlpb} is equivalent to the observability inequality for the solution of the adjoint system \eqref{adjpb} given by the following result.
\begin{theorem}\label{thmobs} 
We assume that $T>2T_*$ and $\delta>d$. Then there exists a constant $C>0$ such that for all $(z_0,z_{0,\Gamma},z_1,z_{1,\Gamma})\in \mathcal{E}\times \mathcal{H}$, the solution of \eqref{adjpb} satisfies the observability inequality
\begin{align}
    \frac{1}{2} &\int_{\Omega}\left(\left|z_1(x)\right|^2+d|\nabla z_0(x)|^2\right) \d x +\frac{1}{2} \int_{\Gamma^1}\left(\left|z_{1,\Gamma}(x)\right|^2+\delta\left|\nabla_{\Gamma} z_{0,\Gamma}(x)\right|^2\right) \d S 
    \leq C \int_{\gamma\times (0,T)} |\pnu z|^2 \d S \d t.
\end{align}
\end{theorem}
The proof of the above theorem follows from our Carleman estimate (Theorem \ref{Thm:Carleman}) by adopting the same strategy of \cite[Proposition 1]{HIY20}, so we omit the details.

\begin{remark}
Theorem \ref{thmobs} extends and improves \cite[Theorem 2.2]{BDEM22} by giving a sharp lower bound for the time required for observability (see \cite[Section 5.2]{BDEM22}). For instance, if $\Omega_1=B_{1}$, then $\psi_0(x)=|x|^2$, $C^\prime=2$, $\underset{x\in\overline{\Omega}}{\min}\,|x|^2=1$, and we can choose $\rho=1$. Hence, we obtain the observability with an explicit minimal time
$$2T_*=\frac{2}{\min\left(\sqrt{d},\sqrt{\frac{2d(\delta-d)}{8\delta}}\right)}\left(\underset{x\in\overline{\Omega}}{\max}\,|x|^2-1\right)^{\frac{1}{2}}.$$
\end{remark}

\begin{remark}\label{rmkobs}
In the case when $\delta<d$, it has been shown in \cite[Theorem 2.4]{BDEM22} that the corresponding observability inequality fails at any time $T > 0$. The assumption $\delta>d$ has also been considered in \cite{MM'2023} for the controllability of a Schr\"odinger equation with a dynamic boundary condition. See open problem 6.1 for further discussion.
\end{remark}

\begin{remark}
It should be emphasized that the strategy we used for observability gives a sharper minimal time than the one obtained by the multiplier method. We refer to \cite{FLL23} for recent results in this context.
\end{remark}

The main result of this section reads as follows.
\begin{theorem}
We assume that $T>2T_*$ and $\delta>d$. Then the system \eqref{ctrlpb} is exactly controllable in time $T$ with a control acting on $\gamma$, i.e., for all initial data $(y_0,y_{0,\Gamma},y_1,y_{1,\Gamma})\in \mathcal{H} \times \mathcal{E}^{-1}$, there exists a control function $v\in L^2(\gamma\times (0,T))$ such that the solution of system \eqref{ctrlpb} satisfies
$(y(T), y_\Gamma(T))=(0,0) \text{ and } \left(\pt y(T), \pt y_{\Gamma}(T)\right)=(0,0) \text{ in } \Omega \times \Gamma^1.$
\end{theorem}
The above controllability result is sharp in the sense that only one control force on the Dirichlet portion is used, in contrast to \cite[Section 4]{Tebou2017}, where two controls are used. This fact has already been observed in a similar situation \cite{BT18}, when $\Omega=\displaystyle\prod\limits_{i=1}^{N} (0,l_{i}),\ l_{i} > 0$. We also refer to \cite{Ma14} when $\Omega$ is a polygon (or polyhedron).

\section{Open problems}\label{sec6}
This section is devoted to some relevant open problems that deserve further investigation.

\subsection*{Open problem 6.1}\label{op1}
In Theorems \ref{Thm:Carleman} and \ref{thm:stab}, we could absorb the term
$\displaystyle
s\lambda \int_{\Gamma^1_T} \varphi |\nabla_\Gamma y_\Gamma|^2\mathrm{d}S\mathrm{d}t    
$
thanks to the assumption $\delta > d$ (see Remark \ref{rmkddelta}). Given the results of \cite[Section 2.2]{BDEM22}, the assumption $\delta > d$ is more involved in such problems. Indeed, by Remark \ref{rmkobs}, a sharp Carleman estimate like \eqref{carleman} (i.e., without tangential and time derivatives on the right-hand side, in contrast to \cite[Theorem 2.2]{Tebou2017}) for the case $\delta < d$ is impossible. However, in the case $\delta=d$, a sharp observability inequality is still an open question.

\subsection*{Open problem 6.2}
All previous Carleman estimates for the wave equations with dynamic boundary conditions incorporate a Dirichlet boundary condition on one portion of the boundary. However, a Carleman estimate for the wave equation with dynamic boundary conditions on the whole boundary $\Gamma$ is still an open problem to be resolved. Namely, the following controllability problem 
\begin{align}
	\label{ctrlpbw}
	\begin{cases}
		\ptt y-d\Delta y + q_{\Omega}(x,t) y=0, &\text{ in } \Omega \times (0,T),\\
		\ptt y_\Gamma -\delta \Delta_\Gamma y_\Gamma + d\pnu y + q_{\Gamma}(x,t) y_\Gamma=\mathds{1}_{\Gamma_*} v, \quad y_\Gamma = y_{\mid_{\Gamma}}, &\text{ on }\Gamma\times (0,T) ,\\
		(y(\cdot,0),y_\Gamma (\cdot,0))=(y_0,y_{0,\Gamma}), \quad(\pt y(\cdot,0),\pt y_\Gamma (\cdot,0))=(y_1,y_{1,\Gamma}), &\text{ in }\Omega\times \Gamma,
	\end{cases}
\end{align}
assuming the existence of $x_0\notin \overline{\Omega}$ such that the control support satisfies $\Gamma_* \supset \{x\in \Gamma\colon (x-x_0)\cdot \nu(x) \ge 0\}$, is an open problem. The classical weight functions (see \eqref{weq}) that are employed for static boundary conditions are not well-adapted for dynamic boundary conditions. Intuitively, one needs a more appropriate weight function $\psi$ satisfying the pseudo-convexity property with respect to both differential operators in $\Omega$ and on $\Gamma$, i.e., there exists $\rho > 0$ such that 
\begin{equation}
    \nabla^2 \psi(\xi, \xi)(x) \geq 2\rho |\xi|^2  \text{ in }\overline{\Omega} \quad \text{ and } \quad \nabla^2_\Gamma \psi (\xi, \xi)(x) \geq 2\rho |\xi|^2 \text{ on }\Gamma\quad \text{ for all }\xi \in \mathbb{R}^n.
\end{equation}
We refer to \cite[Chapter 6]{BY17} for more details. Moreover, the designed weight function should absorb the tricky term
$\displaystyle
s\lambda \int_{\Gamma_{T}}{e^{2s\varphi} \varphi \pnu \psi |\pnu y|^2}\mathrm{d}S\mathrm{d}t,$
which causes a main difficulty in such problems. Such a weight function has not been considered before, up to our knowledge. This issue has recently been resolved in the parabolic case \cite{CGKM23} by regularity estimates. Thus, we believe that this open question deserves more attention, in contrast to \cite[Remark 3]{Tebou2017}.

\section*{Acknowledgment}
We would like to thank the anonymous referees for their invaluable comments and corrections that led to this improved version.

\bibliographystyle{siamplain}

\end{document}